\documentclass[reqno]{amsart}
\usepackage{amsmath,amsthm}
\usepackage{amsfonts}
\usepackage{amssymb}
\usepackage{amsaddr}

\usepackage{hyperref}

\usepackage[square,comma]{natbib}

\theoremstyle{plain}
\newtheorem{theorem}{Theorem}
\newtheorem{lemma}[theorem]{Lemma}
\newtheorem{corollary}[theorem]{Corollary}
\newtheorem{proposition}[theorem]{Proposition}

\theoremstyle{definition}
\newtheorem{definition}[theorem]{Definition}

\theoremstyle{remark}
\newtheorem{remark}[theorem]{Remark}

\newtheorem{example}[theorem]{Example}

\numberwithin{equation}{section}
\numberwithin{theorem}{section}

\normalfont\upshape
\numberwithin{equation}{section}
\numberwithin{theorem}{section}

\newcommand{\set}[2]{\left\{#1\left|\; #2\right.\right\}}

\newcommand{\Prim}{\operatorname{Prim}}

\begin{document}

\title{\textbf{Structure of unital 3-fields }}
\author{{Steven Duplij}}
\address{Center for Information Technology (WWU IT),
Mathematisches Institut, Westf\"{a}lische Wilhelms-Universit\"{a}t M\"{u}nster,
D-48149 M\"{u}nster, Germany}
\email{douplii@uni-muenster.de, sduplij@gmail.com}
\author{{Wend Werner}}
\address{Mathematisches Institut, Westf\"{a}lische Wilhelms-Universit\"{a}t M\"{u}nster\\
Einsteinstrasse 62, D-48149 M\"{u}nster, Germany}
\email{wwerner@uni-muenster.de}

\begin{abstract}
We investigate fields in which addition requires three summands. These ternary fields are shown to be isomorphic
to the set of invertible elements in a local ring $\mathcal{R}$ having $\mathbb{Z}\diagup 2\mathbb{Z}$ as a residual field.
One of the important technical ingredients is to intrinsically characterize the maximal
ideal of $\mathcal{R}$.
We include a number of illustrative examples and prove that the structure of a finite 3-field is not connected
to any binary field.
\end{abstract}
\maketitle

{\small
\textsc{\tableofcontents}
}

\thispagestyle{empty}

\section*{Introduction}
Most of us seem to be biologically biased towards thinking that
it always requires two in order to generate a third. In mathematics
or physics, however, this idea does not seem to rest on a sound foundation:
The theory of symmetric spaces, for example, is nicely described in terms
of Lie or Jordan triple systems (\cite{chu1,upmeier}; see e.g. \cite{boh/wer2}
for a recent development),
and in physics, higher Lie algebras have come into focus
in \cite{nam0} (for later development see e.g. \cite{ker1,azc/izq})
and were e.g.\ applied to the theory of M2-branes in \cite{bag/lam1}. Ternary Hopf
algebras were introduced and investigated in \cite{dup26}.

In order to illustrate why oftentimes (but not here), ternary algebraic structure does
not introduce new aspects, let us digress somewhat and have a closer look at a simple example,
\emph{commutative ternary groups}.

If $G$ is a set and $m:G^3\to G$ is a mapping, what properties should m have in order that $G$
be called a \emph{(commutative) ternary group} with multiplication given by $m$?
Here is a list.
\begin{itemize}
  \item \emph{Associativity}, as always, should mean that there is no need for brackets
  when we multiply several times in a row.
  \item The ternary multiplication is \emph{commutative}, iff it is invariant under
  any permutation of the factors
  \item Finally, every element $g\in G$ has a \emph{(ternary) inverse} $\overline{g}$ (the
  \emph{quer} element) so that for all $g_0\in G$
  \begin{equation*}
  g_0g\overline{g}=g_0.
  \end{equation*}
\end{itemize}
This is all that is needed in the way of a definition. Note that there is no neutral element,
which nonetheless can be defined by saying
$e\in G$ is neutral, iff $eeg=g$, for all $g\in G$.

The reason that we did not include the existence of a neutral element $e$ in the definition is
that as soon as there is one, $G$ equipped with the product
\begin{equation*}
g\times h:=geh,
\end{equation*}
becomes a binary group (with neutral element $e$
and inverses $g^{-1}=\overline{g}$) such that for all $g_{1,2,3}\in G$
\begin{equation*}
g_1\times g_2\times g_3=g_1g_2g_3.
\end{equation*}
So here, $G$ is trivial in that its ternary product directly comes from a binary one.
\footnote{In binary parlance, units as defined here resemble reflections, and so
no uniqueness of unit elements is to be expected. The resulting binary
groups, though, are all pairwise isomorphic: If, in fact, $e_{1,2}$ are two
units, and if we denote the corresponding binary groups by $G_1$ and $G_2$, repectively, then the mapping $\Phi:G_1\to G_2$, given by $\Phi(g)=e_1e_2g$
is easily seen to be an isomorphism.}

This, however always is the case, albeit for a different reason.

A more general example of this kind is the following. Pick a binary commutative group $G$, fix $g_0\in G$, and let
\begin{equation*}
m(g_1,g_2,g_3)=g_0g_1g_2g_3
\end{equation*}
Then $(G,m)$ is a ternary group with $\overline{g}=g_0^{-1}g^{-1}$
for each $g\in G$.
Also, $(G,m)$ is unital iff $g_0=w^2$ for some $w\in G$ in which
case $e=w^{-1}$ is a unit.

No further examples exist, as we will see next.
For a ternary, commutative group $G$ we consider the set of mappings
\begin{equation*}
P(G)=\set{p_{a,b}:G\to G}{a,b\in G,\ p_{a,b}(g)=abg}.
\end{equation*}
Under composition of mappings, $P(G)$ is a (binary) group
(with unit $p_{g,\overline{g}}$, and $p_{a,b}^{-1}=p_{\overline{a},\overline{b}}$).
Fixing an element $g_0\in G$, we find that the map
\begin{equation*}
\Phi_{g_0}: G\to P(G),\qquad \Phi_{g_0}(g)=p_{g,g_0}
\end{equation*}
is a morphism of $G$ into $P(G)$, when the latter is equipped with
the ternary product $p_1p_2p_3=p_{\overline{g}_0,\overline{g}_0}\circ p_1\circ p_2\circ p_3$,
and we obtain

\medskip

\emph{For any commutative ternary group $G$ there is a binary group $G_0$ as well
as an element $g_0\in G_0$ such that $G$ is isomorphic to the ternary group
defined on $G_0$ through the ternary product
$m(g_1,g_2,g_3)=g_0g_1g_2g_3$
}

\medskip

And we conclude: As far as commutative ternary groups are concerned, the
theory doesn't produce much of a novelty.

In this paper we investigate 3-fields, a structure in which the binary operations
of the classical theory are replaced by ternary ones. There is a marked difference
between addition and multiplication, due in principle to the different way in which addition
and multiplication enter the distributive law. Whereas the multiplicative structure of
higher arity in
rings makes easier contact with binary algebra (see e.g. \cite{lee/but}
or \cite{elg/bre}, for example) largely based on the fact that, as above,
multiplicative pairs can be used in these examples just as we did here for commutative ternary groups, ternary addition
produces phenomena of a more unusual kind and has been, to the knowledge
of the authors, treated less thoroughly. We therefore keep multiplication in the fields
binary for the moment and stick to ternary addition. Technically, this fact is hidden
behind the expression \emph{unital}, as in ternary group theory one can very well dispose
of a unit, and, even more strikingly, the truly ternary case is characterized by the
absence of one.

It turns out that there are a number (actually, one that might turn out to be too
large) of interesting examples of 3-fields, finite ones, a certain subset of the 2-adic
numbers, a class of finite skew-3-fields, based on the quaternion group, or a number
of group 3-algebras which actually turn out to be 3-fields.

Here is what we will do in the following: The first section collects some basic
theory (based on the pioneering papers \cite{dor3,pos}),
section 2 introduces the main technical tool that permits a connection
to binary algebra, in the third section we deal with ideals which probably provide the
most uncommon definition in this paper, section 4 is brief on 3-vector spaces
and 3-algebras, just enough in order to be well equipped for a first attack on the
classification of finite 3-fields in the final section. Among other things, we will prove
here that the number of elements of a finite 3-field is a power of two, that their structure
is governed by certain polynomials, with coefficients from the unit disk $\mathcal{B}_2$ of the
2-adic number field
$\mathbb{Q}_2$ and (in the case of a single generator, mapping $\partial \mathcal{B}_2$
into the interior of $\mathcal{B}_2$).
Furthermore, each such field carries a structure totally different from classical fields, because
essentially none of the finite ternary unital 3-fields embeds into a binary field, when the
latter is supposed to carry its canonical ternary structure.

\section{Some history, basics and examples}
The present topic has some precursors. Besides the ones mentioned in the previous
section, groups with additional, quite general "multiple operators" were considered in
abstract form in the late 50's and 60's (with \cite{hig} probably being a first reference for
this perspective; see also the survey \cite{kur1}).
Our construction is connected with the notion of $\left(  n,m\right)  $-rings which was
introduced in \cite{cup} and further studied in
\cite{cel,cro/tim}. The Post theorem for $\left(
n,m\right)  $-rings was formulated in \cite{cro1}.
In almost all of these investigations, the authors treat the case of $n$-fold sums and products;
in many respects, though, the case where $n=3$, shows some more particular features so that we, for
now, restrict our attention to this special case.

First we recall the general notion of a $\left(  3,3\right)  $-ring
\cite{cro1,cel}. We have two different operations on a set $\mathcal{R}$: the ternary
addition $\nu:\mathcal{R}\times \mathcal{R}\times \mathcal{R}\rightarrow \mathcal{R}$ and the ternary multiplication
$\mu:\mathcal{R}\times \mathcal{R}\times \mathcal{R}\rightarrow \mathcal{R}$. We suppose that both operations are
totally associative
\begin{align}
\nu\left(  \nu\left(  x,y,z\right)  ,t,u\right)   &  =\nu\left(  x,\nu\left(
y,z,t\right)  ,u\right)  =\nu\left(  x,y,\nu\left(  z,t,u\right)  \right)  ,\\
\mu\left(  \mu\left(  x,y,z\right)  ,t,u\right)   &  =\mu\left(  x,\mu\left(
y,z,t\right)  ,u\right)  =\mu\left(  x,y,\mu\left(  z,t,u\right)  \right)  ,
\end{align}
where $x,y,z,t,u\in \mathcal{R}$. This means that both $\left\langle \mathcal{R},\nu\right\rangle
$ and $\left\langle \mathcal{R},\mu\right\rangle $ are ternary semigroups. The
connection between them is given by a ternary analog of the distributive law.
A general form of the ternary distributivity (which we will use here) is
\begin{align}
\mu\left(  \nu\left(  x,y,z\right)  ,t,u\right)   &  =\nu\left(  \mu\left(
x,t,u\right)  ,\mu\left(  y,t,u\right)  ,\mu\left(  z,t,u\right)  \right)  ,\\
\mu\left(  t,\nu\left(  x,y,z\right)  ,u\right)   &  =\nu\left(  \mu\left(
t,x,u\right)  ,\mu\left(  t,y,u\right)  ,\mu\left(  t,z,u\right)  \right)  ,\\
\mu\left(  t,u,\nu\left(  x,y,z\right)  \right)   &  =\nu\left(  \mu\left(
t,u,x\right)  ,\mu\left(  t,u,y\right)  ,\mu\left(  t,u,z\right)  \right)  .
\end{align}

The semigroup $\left\langle \mathcal{R},\nu\right\rangle $ is assumed to be a ternary
group so that for all $a,b,c\in \mathcal{R}$ there exists a unique solution of the
equation \cite{dor3,pos},
\begin{equation}
\nu\left(  a,b,x\right)  =c. \label{vab}
\end{equation}
Here, the important notion of a \textit{querelement} \cite{dor3},
denoted by $\tilde{x}$ for the addition $\nu$ comes into play. It satisfies
\begin{equation}
\nu\left(  y,x,\tilde{x}\right)  =x,\label{nx}
\end{equation}
for all $y\in\mathcal{R}$. D\"ornte shows in his paper that the existence of
querelements for all $x\in\mathcal{R}$ is equivalent to unique solvability (\ref{vab}, his axiom $P_3$), and, in particular, the querelement is uniquely determined.

\begin{definition}
A set $\mathcal{R}$ with two operations $\nu$ and $\mu$ satisfying distributivity and for
which $\left\langle \mathcal{R},\nu\right\rangle $ is a (commutative) ternary group and
$\left\langle \mathcal{R},\mu\right\rangle $ is a ternary semigroup is called a $\left(
3,3\right)  $-\emph{ring}, or for shortness, a $3$-\emph{ring}.
\end{definition}

\begin{definition}
If ternary multiplication $\mu$ on $\mathcal{R}$ is commutative, i.e. if
$\mu=\mu\circ\sigma$, where $\sigma$ is any permutation from $S_{3}$, then we
call $\mathcal{R}$ a \emph{commutative} 3-\emph{ring}.
\end{definition}

\begin{definition}
An element $0$ of a 3-ring $\mathcal{R}$, is called
a \emph{ternary zero} iff
\begin{equation}
\mu\left(0,x,y\right)=\mu\left(x,0,y\right)=\mu\left(x,y,0\right)=0
\end{equation}
for all $x,y\in\mathcal{R}$.
\end{definition}

Distributivity in $\mathcal{R}$ shows that $\nu\left(0,0,x\right)=x$ for all $x\in\mathcal{R}$.
Furthermore, other than a mere neutral element for 3-group, a zero element of a 3-ring is uniquely
determined.

\begin{definition}
Let $\mathcal{R}$ be a $\left(  3,3\right)  $-ring. $\mathcal{R}$ is called a
$\left(  2,3\right)  $-\emph{ring}, if its addition $\nu$ is derived from a binary
addition $+$, i.e. $\nu\left(  x,y,z\right)  =x+y+z$. Similarly, it is called
a $\left(  3,2\right)  $-\emph{ring}, if its multiplication $\mu$ is derived from a
binary multiplication $\cdot$, i.e. $\mu\left(  x,y,z\right)  =x\cdot y\cdot
z$.
\end{definition}

\begin{proposition}
Suppose that $\mathcal{R}$ is a $\left(3,3\right)$-ring.
\begin{enumerate}
\item
Under the assumption that $\mathcal{R}$ contains a multiplicative unit $1$,
\begin{equation}
a\bullet b:=\mu\left(a,1,b\right),
\end{equation}
yields a binary associative, commutative product so that
$\mu\left(a,b,c\right)=a\bullet b\bullet c$,
and $\mathcal{R}$ may be viewed as a $\left(3,2\right)$-ring.
\item
If $\mathcal{R}$ contains a zero element $0$, then, similar as in the
previous statement,
\begin{equation*}
a+b:=\nu\left(a,0,b\right)
\end{equation*}
defines a binary, associative and commutative composition with $\nu\left(a,b,c\right)=a+b+c$,
and $\mathcal{R}$ is a $\left(2,3\right)$-ring.
\item
Whenever $\mathcal{R}$ contains both, $1$ and $0$, then it is a binary ring.
\end{enumerate}
\end{proposition}
The proof of this statement (which in large part is similar to the corresponding
one for commutative ternary groups as in the introduction), which we do not want to present here, essentially
consists of writing out all the involved definitions and observing
that the above statements indeed are correct. Note, however, that it is important to have a zero elment
in $\mathcal{R}$ (as opposed to a merely additive neutral element), as this is needed in order
to show that the binary addition satisfies the distributive law.
\begin{definition}
We call $\mathcal{R}$ a \emph{proper} $\left(  3,3\right)  $-\emph{ring}, iff neither $\mu$
nor $\nu$ are derived, and $\mathcal{R}$ is a \emph{proper unital} $3$-\emph{ring}, iff (only) its
multiplication $\nu$ is derived.
\end{definition}

We now come to the central defintion of this paper, the concept of a 3-field.
Fields of this type seem to have first appeared in \cite[Defintiion 3.1]{cro/tim} and
\cite[Defintiion 3.5]{lee/but}. In these papers,
a zero element $z$ is allowed to be contained in a commutative 3-field (so that if a
multiplicative unit exists, the binary situation is covered as well). The aim of the following
is to focus on the less familiar situation and to require, right from the start,
that no zero exists in a unital 3-field.

\begin{definition}
A $3$-ring is called a \emph{$3$-field} iff it is a group with respect to multiplication,
and it is called \emph{unital} iff it contains a multiplicatively neutral element
(so that its multiplicative group is derived).
\end{definition}

Stated differently, a unital 3-field simultaneously carries the structure of a binary multiplicative
and ternary additive group which have to cooperate through the distributive law.
The reader should also note that the requirement that each element possesses a multiplicative inverse
excludes the possibility that the underlying additive ternary group contains a zero element.
It is possible, though, that there are neutral elements for the ternary additive group underlying
a unital 3-field. The fields containing such an element are precisely the ones of characteristic 1
(see Proposition \ref{char-1-char}).

\begin{example}
The simplest example for a unital 3-field in which every element is additively neutral
is $\{1,x\}$ where (necessarily) $x^2=1$. Note that it is only required here that one of these
elements is additively neutral. Furthermore, there is only one more unitel 3-field with 2 elements,
the field $\{1,3\}$ inheriting its structure from $\mathbb{Z}\diagup4\mathbb{Z}$ (see the example
below).
\end{example}
\begin{example}
There are $(2,3)$ fields without a (multiplicative) unit:
$i\mathbb{R}$ becomes a $\left(2,3\right)$-field, when equipped with binary addition
and multiplication inherited from the complex number field.
\end{example}

\begin{example}
The existence of a zero element in (3,3)-fields is not automatic, even if neutral elements abound:
Let $\Phi_2=\{x,y\}$ and assume both elements are additively,
as well as multiplicatively neutral. This yields the results of all possible ternary products
and additions. It is also straightforward to check that the distributive law holds. But
neither $x$ nor $y$ is a zero element for this field. Note that also any cartesian
power of $\Phi_2$ is a (3,3) field in which all elements are additively and multiplicatively
neutral. Using the fundamental theorem
of finitely generated abelian groups one can see that each field with this property
necessarily has $2^k$ elements.
\end{example}

\begin{example}
\label{ex-odd}For a more general construction, start with a unital
3-field $\mathbb{F}$. Fix a unital 3-subfield $\mathbb{F}_1$ as well as an element
$t\in \mathbb{F}\setminus \mathbb{F}_1$
so that $t^2\in \mathbb{F}_1$. Then $t\mathbb{F}_1$ is a $\left(3,3\right)$-field in
which none of the algebraic operations is derived.
\end{example}

\begin{example}
A  set of finite unital $3$-fields is given by
\begin{equation}
\left(  \mathbb{Z}\diagup2^{n}\mathbb{Z}\right)  ^{\mathrm{odd}}=\left\{
2k+1\in\mathbb{Z}\diagup2^{n}\mathbb{Z}\mid0\leq2k+1\leq2^{n}\right\}  .
\end{equation}
The fact that each element has a multiplicative inverse follows from the fact
that $\gcd\left(a,2^{n}\right)=1$, for all
$a\in\left(\mathbb{Z}\diagup2^{n}\mathbb{Z}\right)^{\mathrm{odd}}$.
\end{example}

Recall that a cancellative and commutative $3$-ring $\mathcal{R}$ is called a
$3$-integral domain \cite{cro/tim}.

\begin{example}[$3$-field of fractions \cite{cro/tim}]\label{field of fractions}
For any $3$-integral domain in which neither ternary addition nor multiplication are
derived the $3$-field of fractions as defined in \cite{cro/tim} is a proper $3$-field.
For instance, starting with
\[
\mathbb{Z}^{\mathrm{odd}}=\left\{  2k+1\mid k\in\mathbb{Z}\right\}  ,
\]
we arrive at the proper $3$-field
\begin{equation}
\mathbb{Q}^{\mathrm{odd}}=\left\{  r\in\mathbb{Q}\mid\exists p,q\in
\mathbb{Z}^{\mathrm{odd}},r=\dfrac{p}{q}\right\}  . \label{qodd}
\end{equation}

\end{example}

Trying to find a completion of $\mathbb{Q}^{\mathrm{odd}}$ which itself is a
proper 3-field one has to avoid a zero element in the process. The easiest way
to do this seems to be to exploit the relationship of $\mathbb{Q}
^{\mathrm{odd}}$ with the field of dyadic numbers, $\mathbb{Q}_{2}$. Recall the
definition of the absolute value $\left\vert \cdot\right\vert _{2}$. If $\dfrac
{p}{q}=2^{r}\dfrac{p_{0}}{q_0}$, where neither of the integers $p_0$ and $q_0$
is divisible by 2, we have
\begin{equation}
\left\vert \dfrac{p}{q}\right\vert _{2}=2^{-r}.
\end{equation}
Completion of $\mathbb{Q}$ w.r.t. $\left\vert \cdot\right\vert _{2}$ results
in the field $\mathbb{Q}_{2}$, the elements of which can be formally written as
\begin{equation}
x=\sum_{r\geq-n_{0}}\varepsilon_{r}2^{r},\ \ \ \varepsilon_{r}\in
\mathbb{Z}\diagup2\mathbb{Z}.
\end{equation}
and $\left\vert x\right\vert _{2}=2^{n_{0}}$. Then $\left\vert x\right\vert
_{2}=1$, iff
\begin{equation}
x=1+\sum_{r=1}^{\infty}\varepsilon_{r}2^{r},\ \ \ \varepsilon_{r}\in
\mathbb{Z}\diagup2\mathbb{Z}.
\end{equation}

\begin{example}
The set $\mathbb{Q}_{2}^{\mathrm{odd}}=\left\{  x\in\mathbb{Q}_{2}
\mid\left\vert x\right\vert _{2}=1\right\}  $ is a unital $3$-field w.r.t
multiplication and ternary addition inherited from $\mathbb{Q}_{2}$. This
field is the completion of $\mathbb{Q}^{\mathrm{odd}}$ w.r.t. $\left\vert
\cdot\right\vert _{2}$. Note that this 3-field is compact. Furthermore,
similar to the binary case, $\mathbb{Q}_{2}^{\mathrm{odd}}$ is an inverse
limit $\mathbb{Q}_{2}^{\mathrm{odd}}=\lim\limits_{\longleftarrow}\left(
\mathbb{Z}\diagup2^{n}\mathbb{Z}\right)  ^{\mathrm{odd}}$.
\end{example}

\section{Pairs}

One of the striking differences of the ternary in comparison to to the binary theory is the
fact that for 3-fields there exists a meaningful ideal theory (which will be
initiated in the following section). Central to this concept in classical theories
is the 1-1 correspondence between ideals and kernels of morphisms, a
relationship seemingly doomed in the absence of a zero element.

The purpose of the present section is to find a zero element for a unital 3-field
$\mathbb{F}$, not too far away from $\mathbb{F}$, by embedding it into a
binary unitary ring, in an essentially unique way, and
careful enough so that the essential structure of $\mathbb{F}$ is preserved.

The starting point is the following definition in which we will strip down the
notation and use plus signs even at places where we deal with a ternary
addition. As it will turn out, this might be justified by the fact that the ternary
addition of any 3-field actually is induced by the binary addition of a canonical
binary ring.

\begin{definition}
Let $\mathcal{R}$ be a unital $3$-ring, and for $a,b\in\mathcal{R}$ let
\begin{equation*}
q_{a,b}: \mathcal{R}\to\mathcal{R},\quad
x\longrightarrow x+a+b
\end{equation*}
We will furthernmore use the notation
\begin{equation*}
\mathcal{Q}\left(  \mathcal{R}\right)  =\left\{q_{a,b}\mid a,b\in\mathcal{R}\right\}
\quad\text{and}\quad
\mathcal{U}\left(\mathcal{R}\right)  =\mathcal{Q}\left(  \mathcal{R}\right)  \cup\mathcal{R}
\end{equation*}
\end{definition}

Note that the above definition yields an equivalence relation on the Cartesian
product $\mathcal{R}\times\mathcal{R}$. Similar constructions go back at least
as far as \cite{pos}, serving to \textquotedblleft reduce
arity\textquotedblright.

First, we convert the set of pairs into a (binary) ring. In order to reduce
the technical effort we write pairs in their \emph{standard forms}:
For each pair we have
\[
q_{a,b}=q_{a+b-1,1},
\]
and whenever $q_{s,1}=q_{t,1}$ then $s=t$. With this
notation we (well-)define \textit{binary addition} $+_{q}$ and the
\textit{binary product} $\times_{q}$ for pairs through
\begin{align}
q_{\alpha,1}+_{q}q_{\beta,1}  &  =q_{\alpha
+\beta+1,1},\\
q_{\alpha,1}\times_{q}q_{\beta,1}  &  =q
_{\alpha+\beta+\alpha\beta,1}.
\end{align}

We extend these operations to $\mathcal{U}\left(  \mathcal{R}\right)  $. For
$u,v\in\mathcal{U}\left(  \mathcal{R}\right)  $ let
\begin{align}
u+_{\mathcal{U}}v  &  =\left\{
\begin{array}
[c]{c}
q_{u,v},\ \ \ u,v\in\mathcal{R},\\
u\left(  v\right)  =a+b+v,\ \ \ u=q_{a,b}\in\mathcal{Q}\left(
\mathcal{R}\right)  ,v\in\mathcal{R},\\
u+_{q}v,\ \ \ \ \ u,v\in\mathcal{Q}\left(  \mathcal{R}\right)  ,
\end{array}
\right. \\
u\times_{\mathcal{U}}v  &  =\left\{
\begin{array}
[c]{c}
uv,\ \ \ u,v\in\mathcal{R},\\
q_{av,bv},\ \ \ u\in\mathcal{Q}\left(  \mathcal{R}\right)
,v\in\mathcal{R},\\
u\times_{q}v,\ \ \ \ \ u,v\in\mathcal{Q}\left(  \mathcal{R}\right)  ,
\end{array}
\right.
\end{align}

These operations are well-defined and we furthermore have

\begin{theorem}\label{U(R)}
\label{theor-uq}$\left\langle \mathcal{Q}\left(  \mathcal{R}\right)
,+_{q},\times_{q}\right\rangle $ and $\left\langle \mathcal{U}\left(
\mathcal{R}\right)  ,+_{u},\times_{u}\right\rangle $ are binary rings,
$\mathcal{U}\left(  \mathcal{R}\right)  $ is unital of which $\mathcal{R}$ is
a subring, whereas $\mathcal{Q}\left(  \mathcal{R}\right)  $ is an ideal.
\end{theorem}

For the somewhat lengthy proof it is very convenient to use pairs in their standard forms, and
then there are no major obstacles, although, due to the definition of the algebraic operations,
quite a number different cases have to be distinguished.

\begin{example}
Let $\mathcal{R}=\mathbb{Z}^{\mathrm{odd}}$. Then $\varphi:$ $\mathcal{Q}
\left(  \mathcal{R}\right)  \rightarrow\mathbb{Z}^{\mathrm{even}}$,
$q_{a,b}\overset{\varphi}{\longmapsto}a+b$, is a well-defined
isomorphism of binary (nonunital) rings, and $\mathcal{U}\left(
\mathcal{R}\right)  $ equals $\mathbb{Z}$. Similarly, for the unital $3$-field
$\mathbb{Q}^{\mathrm{odd}}$ we have
\begin{equation}
\mathcal{Q}\left(  \mathbb{Q}^{\mathrm{odd}}\right)  =\mathbb{Q}
^{\mathrm{even}}=\left\{  r\in\mathbb{Q}\mid\exists p\in\mathbb{Z}
^{\mathrm{even}},q\in\mathbb{Z}^{\mathrm{odd}},r=\dfrac{p}{q}\right\}  ,
\end{equation}
as well as
\begin{equation}
\mathcal{U}\left(  \mathbb{Q}^{\mathrm{odd}}\right)  =\mathbb{Q}
^{\mathrm{even}}\cup\mathbb{Q}^{\mathrm{odd}}=\left\{  r\in\mathbb{Q}
\mid\exists p\in\mathbb{Z},q\in\mathbb{Z}^{\mathrm{odd}},r=\dfrac{p}
{q}\right\}  .
\end{equation}

In the same vein,
\begin{equation}
\mathcal{Q}\left(  \left(  \mathbb{Z}\diagup2^{n}\mathbb{Z}\right)
^{\mathrm{odd}}\right)  =\left(  \mathbb{Z}\diagup2^{n}\mathbb{Z}\right)
^{\mathrm{even}},\ \ \ \ \ \mathcal{U}\left(  \left(  \mathbb{Z}\diagup
2^{n}\mathbb{Z}\right)  ^{\mathrm{odd}}\right)  =\left(  \mathbb{Z}
\diagup2^{n}\mathbb{Z}\right)  .
\end{equation}

\end{example}

For the following observations it is convenient to use categorical language.
We denote by $\mathfrak{F}_{3}$ the category in which the objects are unital $3$-fields
and where the morphisms are given by mappings between those which respect the additive
and the multiplicative  structure as well as the unit element. For each
morphism between unital 3-fields $\phi:\mathbb{F}^{\left(  1\right)  }\rightarrow\mathbb{F}^{\left(
2\right)  }$ define mappings $\mathcal{Q}\phi:\mathcal{Q}\mathbb{F}^{\left(
1\right)  }\rightarrow\mathcal{Q}\mathbb{F}^{\left(  2\right)  }$ and
$\mathcal{U}\phi:\mathcal{U}\mathbb{F}^{\left(  1\right)  }\rightarrow
\mathcal{U}\mathbb{F}^{\left(  2\right)  }$, by
\begin{align}
\mathcal{Q}\phi\left(  q_{a,b}\right)   &  =q_{\phi\left(
a\right)  ,\phi\left(  b\right)  },\\
\mathcal{U}\phi\left(  u\right)   &  =\left\{
\begin{array}
[c]{c}
\mathcal{Q}\phi\left(  u\right)  ,\ \ \ u\in\mathcal{Q}\mathbb{F}^{\left(
1\right)  },\\
\phi\left(  u\right)  ,\ \ \ u\in\mathbb{F}^{\left(  1\right)  }.
\end{array}
\right.
\end{align}

It is then requires only to invoke the involved defintions to see that
$\mathcal{Q}\phi$ and $\mathcal{U}\phi$ are unital
morphisms (with respect to binary addition and multiplication defined on these
rings) so that we have defined a functor $\mathcal{U}$ from the category
$\mathfrak{F}_{3}$ to the category of binary unital rings $\mathfrak{R}_{2}$.

Our next goal is to show that the embedding of a unital ternary field $F$
into $U(F)$ is essentially the only way of embedding it into a binary unital ringt,
in the following sense.

\begin{theorem}[Universality theorem]
Suppose $\mathcal{R}$ is a unital binary ring, equipped with the induced structure
of a ternary unital ring, $\mathbb{F}$
is a unital $3$-field, $\varphi:\mathbb{F}\mathbb{\rightarrow}\mathcal{R}$
is a morphism of unital $3$-rings, and write $i_{F}$ for the embedding $\mathbb{F}
\mathbb{\rightarrow}\mathcal{U}\left(  \mathbb{F}\right) $.

Then, there exists a morphism $\bar{\varphi
}:\mathcal{U}\left(  \mathbb{F}\right)  \mathbb{\rightarrow}\mathcal{R}$ of
binary rings, such that $\bar{\varphi}\circ i_{F}=\varphi$.
\end{theorem}

\begin{proof} Define for $R$ (which was supposed to carry its induced unital 3-ring
structure) a map $\pi_R: U(R)\to R$ by $\pi_R(s)=r$ if $s=r\in R$ and $\pi_R(s)=1+a$
if $s=q_{1,a}\in \mathcal{Q}(R)$. Then it is straightforward to check that $\pi_R$ is a unital ring
morphism.\footnote{This somewhat tedious proof consists in going through a number
of cases; for example, if $s_{1,2}=r_{1,2}\in R$ then
\begin{equation*}
\pi_R(s_1+s_2)=\pi_R(q_{1,r_1+r_2-1})=r_1+r_2=\pi_R(s_1)+\pi_R(s_2).
\end{equation*}
All other cases are checked similarly.
}
If we now put $\bar{\varphi}=\pi_R U(\phi)$ it follows that $\bar{\varphi}\circ i_{F}=\varphi$.
\end{proof}

In the next result we will have a closer look at those rings that arise as $\mathcal{U}(F)$,
for a unital 3-field $F$. Recall that a (commutative, binary) ring is called a \emph{local ring} iff
it contains  a unique maximal ideal\footnote{The name for these rings comes from one
of their most prominent examples: Consider the set of continuous functions defined in
the neighbourhood of a point $k$ of some compact space $K$. Identifying two
functions that coincide on a neighbourhood of $k$, yields a ring (with structure inherited
from the ring of continuous functions) in which (the equivalence class of) those functions
vanishing at $k$ form a unique maximal ideal. See e.g.\ \cite[p.111]{Jacobson2} for more.}

\begin{theorem}
[$3$-fields and local rings]\
\begin{enumerate}
\item
Let $\mathcal{R}$ be a (unital) local binary
ring with (unique) maximal ideal $\mathcal{J}$ so that $\mathcal{R\diagup
J\cong}\mathbb{Z}\diagup 2\mathbb{Z}$. Then $\mathcal{R\setminus J}$ is a unital $3$-field.
\item
For any unital $3$-field $\mathbb{F}$, there exists a local binary ring
$\mathcal{R}$ with residual field $\mathbb{Z}\diagup 2\mathbb{Z}$ such that $\mathbb{F}
\mathbb{\cong}\mathcal{R\setminus J}$, where $\mathcal{J}$ is the maximal
ideal of $\mathcal{R}$ and $\mathcal{R\setminus J}$ carries the derived
ternary structure inherited from $\mathcal{R}$.
\end{enumerate}
\end{theorem}

\begin{proof}
(1) Let $\pi_{\mathcal{J}}:\mathcal{R}\rightarrow\mathcal{R\diagup J\ }$ be
the quotient map. Suppose $a_{1,2,3}\in\mathcal{R\setminus J}$. Since
$a\in\mathcal{R\setminus J}$, iff $\pi_{\mathcal{J}}\left(  a\right)  $ is in
$\mathbb{Z}\diagup 2\mathbb{Z}$, it follows that $a_{1}+a_{2}+a_{3}\in\mathcal{R\setminus J}
$. It is straightforward to check that $\mathcal{R\setminus J}$ is an additive
$3$-group. Similarly, the product $a_{1}a_{2}\in\mathcal{R\setminus J}$, and
distributivity is satisfied. It remains to show that each $a\in
\mathcal{R\setminus J\ }$ has a multiplicative inverse. Suppose $a$ has no
inverse, by Krull's theorem it is contained in a maximal ideal different from
$\mathcal{J}$, thus contradicting the locality of $\mathcal{R}$.

(2) Let $\mathbb{F}$ be a unital $3$-field. By {Theorem \ref{theor-uq}
} $\mathcal{U}\left(  \mathbb{F}\right)  $ is a binary unital ring. We show
that $\mathcal{Q}\left(  \mathbb{F}\right)  $ is a unique maximal ideal with
$\mathcal{U}\left(  \mathbb{F}\right)  \mathcal{\diagup Q}\left(
\mathbb{F}\right)  =\mathbb{Z}\diagup 2\mathbb{Z}$. Evidently, $\mathcal{Q}\left(
\mathbb{F}\right)  $ is an ideal of $\mathcal{U}\left(  \mathbb{F}\right)  $.
As all elements in $\mathcal{U}\left(  \mathbb{F}\right)  \mathcal{\setminus
Q}\left(  \mathbb{F}\right)  $ are invertible, this ideal has to be maximal.
By the same reason, $\mathcal{Q}\left(  \mathbb{F}\right)  $ is the only
maximal ideal. So $\mathcal{U}\left(  \mathbb{F}\right)  $ is a local ring.

It remains to show that $\mathcal{U}\left(  \mathbb{F}\right)
\mathcal{\diagup Q}\left(  \mathbb{F}\right)  =\mathbb{Z}\diagup 2\mathbb{Z}$. Take
$r\in\mathcal{U}\left(  \mathbb{F}\right)  $. If $r\in\mathbb{F}$, then
$r+\mathcal{Q}\left(  \mathbb{F}\right)  =1+\mathcal{Q}\left(  \mathbb{F}
\right)  $, because $r+\bar{r}+1=1$, and therefore $r\sim1$. If $r\in
\mathcal{Q}\left(  \mathbb{F}\right)  $, then, of course, $r+\mathcal{Q}
\left(  \mathbb{F}\right)  =0+\mathcal{Q}\left(  \mathbb{F}\right)  $, i.e.
$r\sim0$. So there are only two equivalence classes and hence $\mathcal{U}
\left(  \mathbb{F}\right)  \mathcal{\diagup Q}\left(  \mathbb{F}\right)
=\mathbb{Z}\diagup 2\mathbb{Z}$.
\end{proof}

It is not difficult to see that the functor $\mathcal{U}$ actually establishes an
equivalence of the categories of unital local binary rings with residual field
$\mathbb{Z}\diagup 2\mathbb{Z}$ and the category of unital 3-fields.

\begin{example}
In the case of $\mathbb{Q}^{\mathrm{odd}}_{2}$, its local ring is the
valuation ring $\mathcal{O}\left(  \mathbb{Q}_{2}\right)  =\left\{
z\in\mathbb{Q}_{2}\mid\left\vert z\right\vert _{2}\leq1\right\}  $ with
(maximal) evaluation ideal $\mathcal{B}\left(  \mathbb{Q}_{2}\right)
=\left\{  z\in\mathbb{Q}_{2}\mid\left\vert z\right\vert _{2}<1\right\}  $, and
$\mathbb{Q}^{\mathrm{odd}}_{2}=\mathcal{O}\left(  \mathbb{Q}_{2}\right)
\setminus\mathcal{B}\left(  \mathbb{Q}_{2}\right)  $.
\end{example}

Here is another application:

\begin{theorem}\label{characterizing embedding}
For any unital 3-field $\mathbb{F}$ the following are equivalent.

\begin{enumerate}
\item There exists an embedding of $\mathbb{F}$ into a binary field
$\mathbb{K}$, where the latter is supposed to carry its derived ternary structure.

\item $\mathcal{Q}(\mathbb{F})$ is an integral domain.

\item For each $y\neq1$ the equation
\begin{equation}
x+y-xy=1
\end{equation}
has the only solution $x=1$.
\end{enumerate}
\end{theorem}

\begin{proof}
Writing down what it means for $\mathcal{Q}(\mathbb{F})$ to be an integral
domain, using standard forms of pairs, the equivalence of (2) and (3) is easily
seen. Suppose that
$\mathcal{Q}(\mathbb{F})$ is an integral domain, and denote by $\mathbb{K}$
its field of quotients. Define
\begin{equation}
\Psi: \mathbb{F}\to\mathbb{K},\qquad\Psi(x)=\frac{q_{x,x}}{q_{1,1}}
\end{equation}
Then $\Psi$ is injective as $\Psi(x)=\Phi(y)$ is equivalent to
\begin{equation}
q_{x+x,x+x}=q_{x,x}q_{1,1}=q_{y,y}q_{1,1}=q_{y+y,y+y}
\end{equation}
which by how pairs were defined is the same as $x=y$.
The map $\Psi$ is also multiplicative, because $\Psi(x)\Psi(y)=\Psi(xy)$ is
equivalent to
\begin{equation}
q_{xy+xy,xy+xy}q_{1,1}=q_{xy,xy}q_{1+1,1+1}
\end{equation}
where on both sides the same pair is written in a slightly different form.
Additivity is proven in a similar way.

Conversely, whenever there exists an
embedding $\mathbb{F}\to\mathbb{K}$, $\mathcal{Q}(\mathbb{F})$ is injectively
mapped into $\mathcal{Q}(\mathbb{K})$, which is an integral domain.
\end{proof}

\section{Ideals}

Because of the absence of zero in a proper $3$-ring, the usual correspondence
between ideals and kernels of morphisms is no longer available. Instead, we
apply the results of the previous section.

Let us consider a morphism of unital $3$-rings $\phi:\mathcal{R}
_{1}\rightarrow\mathcal{R}_{2}$. Then $\ker\mathcal{U}\left(\phi\right)$
is an ideal of $\mathcal{U}\left(\mathcal{R}_{1}\right)$, and the
underlying equivalence relation on $\mathcal{R}_{1}$ is given by
\begin{equation}
r_{1}\sim r_{2}\Longleftrightarrow\exists q\in\ker\mathcal{U}\left(
\phi\right)  :r_{1}+q=r_{1}.
\end{equation}
Note that $\ker\mathcal{U}\left(\phi\right)$ is contained in
$\mathcal{Q}\left(  \mathcal{R}_{1}\right)$, and so $q$ must be an
additive pair. The above is a motivation for

\begin{definition}
An \emph{ideal} for a unital $3$-ring $\mathcal{R}$ is any (binary) ideal of
$\mathcal{Q}\left(\mathcal{R}\right)$. We furthermore denote the quotient
an ideal of $\mathcal{Q}(\mathcal{R})$ defines on
$\mathcal{R}$ by $\mathcal{R}\diagup\mathcal{J}$.
\end{definition}
Note that an ideal $\mathcal{J}$ intersecting $\mathcal{R}$ would lead to
a zero element in the quotient $\mathcal{R}\diagup\mathcal{J}$ and thus
to a $3$-ring whose additive structure is induced by a binary addition,
a case that we are excluding here.
Defining ideals in this way, the homomorphism theorem for 3-rings has a proof
almost identical to the binary case.
\begin{proposition}
\label{prop-q}Suppose $\mathcal{R}$ and $\mathcal{S}$ are unital 3-rings,
and $\phi:\mathcal{R}\to \mathcal{S}$ is a morphism.
Then the quotient $\mathcal{R}\diagup\ker\mathcal{U}\left(\phi\right)$
is a unital $3$-ring, and
$\mathcal{R}\diagup\ker\mathcal{U}\left(\phi\right)\simeq\operatorname{Im}\phi$.
\end{proposition}

We prefer the expression \textquotedblleft an ideal \textit{for} a unital
$3$-ring\textquotedblright\ over \textquotedblleft an ideal \textit{of}
...\textquotedblright, as the former is not a subset of $\mathcal{R}$.

The following theorem is an analogue to the fact that for a binary
ring the quotient by an ideal is a field, if and only if the ideal is maximal.

\begin{theorem}\label{th-maxi}
For a unital $3$-ring $\mathcal{R}$ and an ideal
$\mathcal{I}\subseteq\mathcal{Q}\left(\mathcal{R}\right)$, the quotient
$\mathcal{R}\diagup\mathcal{I}$ is a unital $3$-field, if and only if for any proper
ideal $\mathcal{J}$ of $\mathcal{U}\left(\mathcal{R}\right)$ for which
$\mathcal{J}\supseteq\mathcal{I}$ it follows that
$\mathcal{J}\cap\mathcal{R}=\varnothing. \label{ji}$
\end{theorem}

\begin{proof}
Suppose that $\mathbb{F}=\mathcal{R}\diagup\mathcal{I}$ is a unital
$3$-field but $\mathcal{J}\cap\mathcal{R}\neq\varnothing$ for some proper
ideal $\mathcal{J}$ containing $\mathcal{I}$. Let $\pi: \mathcal{U}\left(
\mathcal{R}\right)  \rightarrow\mathcal{U}\left(  \mathcal{R}\right)
\diagup\mathcal{I}$ be the quotient map. Note that there is
a 1-1 correspondence between the ideals of
$\mathcal{U}\left(\mathcal{R}\right)\diagup\mathcal{I}$ and
and the ideals of $\mathcal{U}\left(\mathcal{R}\right)$ containing
$\mathcal{I}$, given by
\begin{equation*}
    \mathcal{J}\longrightarrow \pi(\mathcal{J}), \qquad\text{resp.}\qquad
    \mathcal{J}'\longrightarrow \pi^{-1}(\mathcal{J}')
\end{equation*}
(where $\pi\pi^{-1}(\mathcal{J}')=\mathcal{J}'$ and $\pi^{-1}\pi(\mathcal{J})=\mathcal{J}$).
It then follows that  $\pi(\mathcal{J})$ is an ideal in
$\mathcal{U}\left(\mathcal{R}\right)\diagup\mathcal{I}=\mathcal{U}\left(\mathbb{F}\right)$.
Since $\emptyset\neq\pi(\mathcal{J}\cap\mathcal{R})\subseteq\pi(\mathcal{J})\cap\mathbb{F}$,
this ideal contains an invertible element, hence $\pi\left(\mathcal{J}\right)=
\mathcal{U}(\mathbb{F})$, and so, again by the above correspondence, $\mathcal{J}=\mathcal{U}(\mathcal{R})$.

If, on the other hand, for any proper ideal $\mathcal{J}\supseteq
\mathcal{I}$ we have $\mathcal{J}\cap\mathcal{R}=\varnothing$, we choose
$r\in\mathcal{R}\diagup\mathcal{I}$ as well as $r_{0}\in\mathcal{R}$ with
$\pi\left(  r_{0}\right)  =r$. If $r$ were not invertible, then the ideal
$\mathcal{J}_{0} $ generated by $r_{0}$ and $\mathcal{I}$ would be proper,
contain $\mathcal{I}$ and intersect $\mathcal{R}$.
\end{proof}

\begin{example}
\label{ex-32}Consider the $\left(  3,2\right)  $-ring $\mathbb{Z}
^{\mathrm{odd}}=\left\{  2k+1\mid k\in\mathbb{Z}\right\}  $. Note that each
proper ideal in the non-unital ring $\mathbb{Z}^{\mathrm{even}}$ is principal,
i.e.\ they are of the form $\left(  2k_{0}\right)  =\left\{  2k_{0}k\mid
k\in\mathbb{Z}\right\}  $, $k_{0}\in\mathbb{Z}$. We claim that $\left(
2k_{0}\right)  $ satisfies the hypothesis of Theorem (\ref{th-maxi}), iff $k_{0}=2^{n}$, $n\in\mathbb{N}$.
In fact, suppose that $p\mid k_{0}$ and $p\neq2$. Then $\left(  2p\right)
\supseteq\mathcal{I}$ and $p\in\mathbb{Z}^{\mathrm{odd}}$, and so
$\mathcal{I}$ cannot satisfy the maximiality condition from Theorem (\ref{th-maxi}).
\end{example}

\begin{example}
Let $\mathbb{F}$ be a proper unital $3$-field, then each ideal in $\mathcal{U}\left(
\mathbb{F}\right) $ must be contained in $\mathcal{Q}(\mathbb{F})$ and has to
fullfill the hypothesis of Theorem \ref{th-maxi} (what could be called
\textquotedblleft evenly maximal\textquotedblright ),
and so for each ideal $\mathcal{J}$ of $\mathcal{U}\left(  \mathbb{F}\right)  $,
$\mathbb{F}\diagup\mathcal{J}$ again is a field. This is quite different from
the binary case, where fields do not possess non-trivial quotients.
\end{example}

\begin{example}\label{Qodd-ideals}
The proper ideals for $\mathbb{Q}^{\mathrm{odd}}=\left\{  r\in\mathbb{Q}
\mid\exists p,q\in\mathbb{Z}^{\mathrm{odd}},r=\dfrac{p}{q}\right\}  $ are of
the form
\begin{equation}
\mathcal{J}_{n}=\left\{  r\in\mathcal{U}\left(  \mathbb{Q}^{\mathrm{odd}
}\right)  \mid\exists q\in\mathbb{Z}^{\mathrm{odd}},\exists u\in
\mathbb{Z},r=2^{n}\dfrac{u}{q}\right\}=\langle 2^{n}\rangle,\qquad n\in\mathbb{N}.
\end{equation}
Obviously, all the $\mathcal{J}_{n}$ are ideals for $\mathbb{Q}
^{\mathrm{odd}}$. Conversely, let $\mathcal{J}$ be an ideal, and
\begin{equation}
n_{0}=\min\left\{  n\in\mathbb{N}\mid\exists p,q\in\mathbb{Z}^{\mathrm{odd}
},2^{n}\dfrac{p}{q}\in\mathcal{J}\right\}  .
\end{equation}
Because any $r\in\mathcal{J}$ is of the form $2^{n}\dfrac{u}{q}$,
$u\in\mathbb{Z}$, $q\in\mathbb{Z}^{\mathrm{odd}}$, $n\geq n_{0}$ we must have
$\mathcal{J}_{n_{0}}\supseteq\mathcal{J}$. Fix an element $2^{n_{0}}
\dfrac{p_{0}}{q_{0}}\in\mathcal{J}$, $p_{0},q_{0}\in\mathbb{Z}^{\mathrm{odd}}
$. Then $2^{n_{0}}\dfrac{u}{q}\in\mathcal{J}$ for all $u\in\mathbb{Z}$,
$q\in\mathbb{Z}^{\mathrm{odd}}$ and hence $\mathcal{J}\supseteq\mathcal{J}
_{n_{0}}$.
\end{example}

We apply this observation to prime fields. Let us consider a unital $3$-field
$\mathbb{F}$ with unit $1$ and define $\mathbb{F}^{\mathrm{prim}}$ to be the
$3$-subfield generated by $1$, i.e. $\mathbb{F}^\mathrm{prim}=\left\langle
1\right\rangle $.

\begin{definition}
The \emph{characteristic} of a unital 3-field is
$\chi\left(\mathbb{F}\right)=\left\vert \mathbb{F}^{\mathrm{prim}}\right\vert$.
\end{definition}

\begin{theorem}
\label{theor-fp}If $\mathbb{F}^{\mathrm{prim}}$ is finite, then there is $n\in\mathbb{N}_0$
so that $\mathbb{F}^\mathrm{prim}\cong
\left(\mathbb{Z}\diagup2^{n}\mathbb{Z}\right)^{\mathrm{odd}}$.
Otherwise, $\mathbb{F}^\mathrm{prim}\cong\mathbb{Q}^{\mathrm{odd}}$.
\end{theorem}

\begin{proof}
Define a morphism $\psi:\mathbb{Q}^{\mathrm{odd}}\rightarrow\mathbb{F}
^{\mathrm{prim}}$ by $\psi\left(  p\diagup q\right)  =pq^{-1}$ which is
well-defined and surjective (since $\operatorname*{im}\psi$ is a
$3$-subfield containing $1$), and so we must prove that $\mathbb{Q}
^{\mathrm{odd}}\diagup\mathcal{J}_{n}=\mathbb{Z}_{2^{n}}^{\mathrm{odd}}$, for
$n\in\mathbb{N}_{0}$. Since the case $\ker\psi=\left\{  0\right\}  $ is
trivial, we suppose $n> 1$. Then division with reminder by $2^{n}$ yields a
morphism $\mathbb{Z}^{\mathrm{odd}}\to\mathbb{Z}_{2^{n}}^{\mathrm{odd}}$, which extends
to the quotient 3-field $\mathbb{Q}^{\mathrm{odd}}$. It then follows that
the kernel of this extension is the ideal $\mathcal{J}_{n}$.
\end{proof}

We conclude this section with a characterisation of fields with characteristic 1.
\begin{proposition}\label{char-1-char}
A unital 3-fields $\mathbb{F}$ is of characteristic 1, iff its additive 3-group
contains a neutral element.
\end{proposition}
\begin{proof}
In a unital 3-field of characteristic 1, the unit always is a neutral element for
the ternary addition (as is any other element). Conversely, if the element 1
is an additive neutral element then, for arbitrary elements $e,x\in\mathbb{F}$
we have
\begin{equation*}
x+e+e=e(e^{-1}x+1+1)=x,
\end{equation*}
and hence, the characteristic of $\mathbb{F}$ equals one.
\end{proof}

\section{3-vector spaces and unital 3-algebras}

We start by defining a ternary analogue of the concept of a vector space.

\begin{definition}
A 3-\emph{vector space} consists of a commutative
3-group of vectors, $V$, a unital 3-field $\mathbb{F}$ as well as an action of
$\mathbb{F}$ on $V$ (so hat $f(gv)=(fg)v$ for all $f,g\in\mathbb{F}$ and $v\in V$).
Furthermore, $1v=v$ for all $v\in V$, and the (ternary analog
of the) standard distributivity laws hold. Linear mappings between
3-vector spaces are defined in the obvious way.
\end{definition}

Given a 3-vector space $V$ over $\mathbb{F}$ we have a canonical action of
$\mathcal{U}\left(  \mathbb{F}\right)  $ on $\mathcal{U}\left(  V\right) $, given
in each of the as yet undefined cases
$\mathbb{F}\times\mathcal{Q}\left(V\right)\to\mathcal{Q}\left({V}\right)$,
$\mathcal{Q}\left( \mathbb{F}\right)\times V\to\mathcal{Q}\left({V}\right)$, and
$\mathcal{Q}\left( \mathbb{F}\right)\times\mathcal{Q}\left(V\right)\to\mathcal{Q}\left({V}\right)$
similar to the definitions that led to Theorem \ref{U(R)}. With these actions in
place, $\mathcal{U}\left(V\right)$ is a binary module over the binary ring $\mathcal{U}(\mathbb{F})$.

\begin{definition}
A subset\ $E\subseteq V$ of a $3$-vector space over a unital $3$-field
$\mathbb{F}$ is called a \textit{generating system}, iff any element of $V$
can be represented as $\sum_{i=1}^{n}\lambda_{i}a_{i}$ with $\lambda_{i}\in
U\left(  \mathbb{F}\right)  $, $a_{i}\in E$, and $\sum_{i=1}^{n}\lambda_{i}
\in\mathbb{F}$. A subset $A$ is called a \textit{basis}, iff this representation is
unique. If $A$ is any subset of $V$ we denote by $\operatorname{lin}A$ the 3-vector
subspace of $V$ generated by $A$.
\end{definition}

\begin{remark}
It is important to observe that any linear combination $\sum_{i=1}^{n}
\lambda_{i}v_{i}$ with $\lambda_{i}\in U\left(  \mathbb{F}\right)  $,
$v_{i}\in V$, $\sum_{i=1}^{n}\lambda_{i}\in\mathbb{F}$ yields an element of
$V$.
\end{remark}

\begin{example}
A $3$-vector space $V$ over the unital $3$-field $\mathbb{F}$ is given by
\begin{equation}
\left(  \mathbb{F}^{n}\right)  ^\mathrm{free}=\left\{  \left(  a_{1},\ldots,
a_{n}\right)  \in\mathcal{U}\left(  \mathbb{F}\right)  \mid\sum_{i=1}^{n}
a_{i}\in\mathbb{F}\right\}.
\end{equation}
It has a basis consisting of elements $e_{i}=\left(  \delta_{ij}\right)
_{j}\in\left(  \mathbb{F}^{n}\right)  ^\mathrm{free}$.
\end{example}

\begin{example}
Note that the Cartesian product $\mathbb{F}^{n}$ (which is a 3-field) is
a 3-vector space as well, which however does not possess a basis if $n$ is
different from 1. If $n=1$, $\mathbb{F}^\mathrm{free}=\mathbb{F}$, and any element of
$\mathbb{F}$ is a basis.
\end{example}

\begin{proposition}
Every $3$-vector space over a unital $3$-field has a free resolution, i.e.\ there is
a 3-vector space $V^\mathrm{free}$ with basis which has $V$ as a quotient.
\end{proposition}

\begin{proof}
We pick a generating set $A$, and let
$V^\mathrm{free}=\left\{\sum_{\alpha\in A_0\text{\ finite} }f_{\alpha}\alpha\mid
\sum f_\alpha\in\mathbb{F}\right\}  $. Define $\phi_{V}:V^\mathrm{free}\rightarrow V$ by
$\sum_{\alpha\in A_0\text{finite} }f_{\alpha}\alpha \longmapsto
\sum_{\alpha\in A_0\text{finite} }f_{\alpha}\alpha$ (so that the formal sum
$\sum_{\alpha\in A_0\text{finite} }f_{\alpha}\alpha$ in $V^\mathrm{free}$
is mapped to a sum in $V$). Then
$\ker\phi_{V}$ is a $U\left(  \mathbb{F}\right)  $-submodule contained in
$\mathcal{Q}\left(  V^\mathrm{free}\right)  $. Hence $V$ is isomorphic to
$V^\mathrm{free}\diagup\ker\phi_{V}$.
\end{proof}

\begin{corollary}
\label{corr-uk}The number of elements of a 3-vector space over the finite
3-field $\mathbb{F}$, generated by n elements is
\begin{equation}
\dfrac{\left\vert \mathcal{U}\left(  \mathbb{F}\right)  \right\vert ^{n}
}{2\left\vert \ker\phi_{V}\right\vert }=2^{n-1}\dfrac{\left\vert
\mathbb{F}\right\vert ^{n}}{\left\vert \ker\phi_{V}\right\vert }.
\end{equation}

\end{corollary}

\begin{proof}
We first note that the number of elements in the free 3-vector space of dimension
$n$ is equal to half the number of elements in $\mathcal{U}\left(\mathbb{F}\right)^n$
and so this statement follows from the above theorem as well as the fact that
$\left\vert \mathcal{U}\left( \mathbb{F}\right)
\right\vert =2\left\vert \mathbb{F}\right\vert $ (which is rapidly seen by using the standard
form of pairs).
\end{proof}

\begin{example}
Consider $\mathbb{F}^{2}=\left\{  1,3\right\}  ^{2}=\left\{  \binom{a}{b}\mid
a,b,\in\left(\mathbb{Z\diagup}4\mathbb{Z}\right)^{\mathrm{odd}}\right\}  $
over $\mathbb{F=}\left(  \mathbb{Z\diagup}4\mathbb{Z}\right)  ^{\mathrm{odd}
}=\left\{  1,3\right\}  $. We have $V=\left\{  \left(
\begin{array}
[c]{c}
1\\
1
\end{array}
\right)  ,\left(
\begin{array}
[c]{c}
3\\
1
\end{array}
\right)  ,\left(
\begin{array}
[c]{c}
1\\
3
\end{array}
\right)  ,\left(
\begin{array}
[c]{c}
3\\
3
\end{array}
\right)  \right\}  $. A generating set is $\left\{  \left(
\begin{array}
[c]{c}
1\\
1
\end{array}
\right)  ,\left(
\begin{array}
[c]{c}
3\\
1
\end{array}
\right)  \right\}  =\left\{  e_{1},e_{2}\right\}  $. Thus, a free resolution
is given by $V^\mathrm{free}=\left\{  ae_{1}+be_{2}\mid a,b\in\mathbb{Z\diagup
}4\mathbb{Z},\ a+b\in\left(  \mathbb{Z\diagup}4\mathbb{Z}\right)
^{\mathrm{odd}}\right\}  $.
\end{example}

\begin{definition}
Let $A$ be a 3-vector space over the unital 3-field $\mathbb{F}$. We call $A$
a \emph{unital (commutative)} 3-\emph{algebra}, iff there exists a binary multiplication
$\left(  \bullet\right)  $ on $A$ so that $\left(  A,+,\bullet\right)  $ is a
(commutative) unital 3-ring.
\end{definition}

In the following, 3-algebras will be mostly commutative.

\begin{example}
Let $G$ be a binary group and $\mathbb{F}$ a unital $3$-field. The group algebra of
$G$ over $\mathbb{F}$ is defined by
\begin{equation}
\mathbb{F}G=\left\{  \phi:G\rightarrow\mathcal{U}\left(  \mathbb{F}\right)
\mid\sum_{g\in G}\phi\left(  g\right)  \in\mathbb{F}\right\}
\end{equation}
together with the convolution product
$\left(\phi\ast\psi\right)\left(g\right)
=\sum_{g_{1}g_{2}=g}\phi\left(g_{1}\right)\psi\left(g_{2}\right)$
which is well-defined, because $\phi\ast\psi\in\mathbb{F}G$. These 3-algebras quite often seem
to be 3-fields. For example, in the case where $G$ equals the additive group
$\mathbb{Z}\diagup n\mathbb{Z}$, we have
$\mathbb{F}G=\mathbb{F}\left(n\right)$, as defined below.
\end{example}

\begin{definition}
Fix a unital 3-field $\mathbb{F}$ and let
\begin{multline}
\mathbb{F}\left[x_{1},\ldots,x_{n}\right]=\\
=\left\{\sum_{\left\vert\alpha\right\vert\leq N}f_{\alpha}x^{\alpha}\mid
\alpha=\left(  \alpha_{1},\ldots,\alpha_{n}\right),\quad N\in\mathbb{N}_{0},
\quad f_{\alpha}\in\mathcal{U}\left(\mathbb{F}\right),
\quad\sum f_{\alpha}\in\mathbb{F}\right\}  .
\end{multline}
We call this space the \emph{polynomial algebra in $n$ variables over the $3$-field
$\mathbb{F}$}. This space is a unital $3$-algebra under the usual product of polynomials.
\end{definition}

We calculate that
\begin{multline}
\mathcal{Q}\left(\mathbb{F}\left[x_{1},\ldots,x_{n}\right]\right)=:
\mathbb{F}\left[x_{1},\ldots,x_{n}\right]^{\mathrm{even}}\\
=\left\{\sum_{\left\vert \alpha\right\vert \leq N}f_{\alpha}x^{\alpha
}\mid\alpha=\left(  \alpha_{1},\ldots,\alpha_{n}\right)  ,\ \ N\in
\mathbb{N}_{0},\ \ f_{\alpha}\in\mathcal{U}\left(  \mathbb{F}\right)
,\ \ \sum f_{\alpha}\in\mathcal{Q}\left(  \mathbb{F}\right)  \right\}  .
\end{multline}
and $\mathcal{U}\left(  \mathbb{F}\left[x_{1},\ldots,x_{n}\right]\right)
=\mathcal{U}\left(\mathbb{F}\right)\left[x_{1},\ldots,x_{n}\right]$.

\begin{theorem}[Universality of polynomial algebras]
The polynomial algebra $\mathbb{F}\left[x_{1},\ldots,x_{n}\right]$
is universal in the class of unital
$3$-algebras over $\mathbb{F}$, generated by $n$ elements.

The polynomial algebra $\mathbb{Q}^{\mathrm{odd}}\left[x_{1},\ldots,x_{n}\right]$
is universal in the class of all unital $3$-algebras over any of the prime fields,
generated by $n$ elements.
\end{theorem}

\begin{proof}
{\bf (1)}
If $A$ is an algebra generated by $a_{1},\ldots,a_{n}$, define $\Psi:\mathbb{F}\left[
x_{1},\ldots,x_{n}\right]  ^{\mathrm{}}\rightarrow A$ by
\begin{equation}
\Psi\left(\sum_{\left\vert \alpha\right\vert \leq N}f_{\alpha}x_{1}^{\alpha_{1}},\ldots,
x_{n}^{\alpha_{n}}\right)  =\sum_{\left\vert \alpha\right\vert \leq
N}f_{\alpha}a_{1}^{\alpha_{1}},\ldots,a_{n}^{\alpha_{n}}
\end{equation}
and use section 2 to
see that $A\cong\mathbb{F}\left[  x_{1},\ldots,x_{n}\right]  ^{\mathrm{}}
/\ker\Psi$.

{\bf(2)}
The statement about
$\mathbb{Q}^{\mathrm{odd}}\left[x_{1},\ldots,x_{n}\right]$
follows by applying the first part of this theorem for the respective prime
field and then by combining the quotient mapping with the one from
Example \ref{Qodd-ideals}.
\end{proof}

\begin{example}\label{Definition F(n)}
Fix a unital 3-field $\mathbb{F}$ as well as a natural number $n$, and define
\begin{equation}
\mathbb{F}\left(  n\right)  =\left\{  P=f+\sum_{i=1}^{n-1}a_{i}\left(
x-1\right)  ^{i}\mid a_{i}\in\mathcal{U}\left(  \mathbb{F}\right)
,\ f\in\mathbb{F},\ \left(  x-1\right)  ^{n}=0\right\}  .
\end{equation}
Here, we use the language of 'generators and relations' so that
requiring $\ \left(  x-1\right) ^{n}=0$ means that we divide out the ideal
generated by this polynomial.
It follows that $\mathbb{F}\left(n\right)$ is a unital 3-algebra, generated by the single
polynomial $x-1$. It actually is a 3-field,
as it is isomorphic to $\mathbb{T}\left(n,\mathbb{F}\right)$, which we introduce in
\end{example}

\begin{example}
The \emph{Toeplitz field} of order $n$ over $\mathbb{F}$,
$\mathbb{T}\left(n,\mathbb{F}\right)$, consists, as set, of all matrices
\begin{equation}
t=\left(
\begin{array}
[c]{cccc}
f & \cdots & 0 & 0\\
b_{1} & \ddots & \ddots & 0\\
\vdots & \ddots & \ddots & \vdots\\
b_{n-1} & \cdots & b_{1} & f
\end{array}
\right),\ \ f\in
\mathbb{F},\ \ b_{i}\in\mathcal{U}\left(  \mathbb{F}\right)  ,\ \ i=1,\ldots,n-1.
\end{equation}
Note that the inverse of each $t$ is of the same form, and hence the Toeplitz
fields are commutative $3$-subfields of the triangular $3$-fields
from Example \ref{ex-tri}. The number of elements in this field is
\begin{equation}
\left\vert \mathbb{T}\left(  n,\mathbb{F}\right)  \right\vert =\left\vert
\mathbb{F}\right\vert \left\vert \mathcal{U}\left(  \mathbb{F}\right)
\right\vert ^{n-1}=2^{n-1}\left\vert \mathbb{F}\right\vert .
\end{equation}
\end{example}

We now can show that $\mathbb{F}\left(n\right)\cong\mathbb{T}\left(n,\mathbb{F}\right)$
(first, as ternary unital rings, and then, consequently, as fields):
The 3-vector space $\mathbb{F}\left(n\right)$ has basis
\begin{equation}
E=\left\{  e_{i}=\left(  x-1\right)  ^{i}\mid i=0,\ldots,n-1\right\}  ,
\end{equation}
and if we consider an element $P=f+\sum_{i=1}^{n-1}b_{i}\left(  x-1\right)
^{i}\in\mathbb{T}_{+}\left(  n,\mathbb{F}\right)  $, $f\in\mathbb{F}
,\ \ b_{i}\in\mathcal{U}\left(  \mathbb{F}\right)  $, $i=1,\ldots,n-1$ as a
linear map on $\mathbb{F}\left(  n\right)  $, then its matrix representation
w.r.t.\ $E$ is given by
\begin{equation}
t_{f}=\left(
\begin{array}
[c]{cccc}
f & \cdots & 0 & 0\\
b_{1} & \ddots & \ddots & 0\\
\vdots & \ddots & \ddots & \vdots\\
b_{n-1} & \cdots & b_{1} & f
\end{array}
\right)  .
\end{equation}
Since the product of $\mathbb{F}\left(  n\right)  $ turns out to be the matrix
product of these matrices, the claim has been proven.

\section{Finite fields}
In general, the theory of the finite unital 3-fields looks quite different from the
binary one, in which all such fields can be explicitly listed.
In the case of unital 3-fields such a classification seems to be out of
reach for the moment, which is due to the fact that the minimal number of generators
for this type of field no longer has to be
one, as in the binary case. In fact, finite unital 3-fields do not share much with
their binary counterparts.

\begin{theorem}
A finite unital 3-field $\mathbb{F}$ admits an embedding into a binary field
$\mathbb{K}$ (equipped with its derived ternary addition) if and only if $\mathbb{F}=\{1\}$.
\end{theorem}

\begin{proof}
If $n=\chi(\mathbb{F})>1$, then $\mathcal{Q}(\Prim\mathbb{F})=(\mathbb{Z}/2^n\mathbb{Z})^{\mathrm{even}}$ is not an integral domain,
and so $\Prim\mathbb{F}$ (and much less $\mathbb{F}$) can be embedded into a
binary field, in light of Theorem\ref{characterizing embedding}. In case $\chi(\mathbb{F})=1$, we
will see in Theorem~\ref{Singly generated fields} how a subfield of $\mathbb{F}$ looks which
is generated by a single element, and the statement then follows as
$\mathcal{Q}(\mathbb{F})$ in this case cannot be an integral domain, either.
\end{proof}

Our second result remarkably pairs with the binary case, though.

\begin{theorem}[Cardinality of finite fields]
For each finite unital $3$-field the number of elements is a power of $2$.
\end{theorem}

\begin{proof}
Clearly, each finite unital 3-field $\mathbb{F}$ is a $3$-vector space over
$\operatorname*{Prim}\mathbb{F}$. By {Corollary \ref{corr-uk}}, the
number of elements in $\mathbb{F}$ is $2^{n-1}\dfrac{\left\vert
\operatorname*{Prim}\mathbb{F}\right\vert ^{n}}{\left\vert \ker\phi
_{\mathbb{F}}\right\vert }=2^{n-1}\dfrac{\chi\left(  \mathbb{F}\right)  ^{n}
}{\left\vert \ker\phi_{\mathbb{F}}\right\vert }$. According to
{Theorem\ref{theor-fp}}, $\chi\left(  \mathbb{F}\right)  $ is a power of 2, which
$\left\vert \ker\phi_{\mathbb{F}}\right\vert $ must divide, and the result follows.
\end{proof}
For any polynomial $P=\sum_\nu a_\nu x^\nu$ in $\mathbb{Q}[x]$ we let
$\|P\|_2=\max_\nu|a_\nu|_2$ (what some authors call the \emph{Gauss norm}).
Then $\|P\|_2\leq 1$ if and only if
$P\in\mathcal{U}(\mathbb{Q}^{\mathrm{odd}})[x]$, and $\|P\|_2=1$ iff,
moreover, $P$
has at least one coefficient in $\mathbb{Q}^{\mathrm{odd}}$.
Then, for any $P,Q\in\mathbb{Q}[x]$, we have
$\|PQ\|_2=\|P\|_2\|Q\|_2$.
This follows from the fact that
the product of two polynomials in $\mathbb{Z}[x]$, having both at least one
odd coefficient, possesses itself at least one odd coefficient (see, for
example, \cite[pp.35--43 ]{bgr}).

The following Lemma follows from these properties of the Gauss norm.
\begin{lemma}\
\label{lem-pid}
The irreducible polynomials for the ring $\mathcal{U}(\mathbb{Q}^{\mathrm{odd}})[x]$,
i.e. those polynomials that cannot be factored
into the product of two polynomials, are:
the constant polynomial 2 and those polynomials $P$
which are irreducible in $\mathbb{Q}[x]$,
and for which $\|P\|_2=1$.
\end{lemma}
\begin{definition}
Let $R$ be a unital 3-ring. We call $p\in \mathcal{Q}(R)$ \emph{completely even}
if and only if in every factorization of $p$ within $\mathcal{U}(R)$, factors are units of $R$ or elements from $\mathcal{Q}(R)$.
\end{definition}
We will be concerned here with the polynomial ring $\mathbb{F}[x]$ in which a polynomial
$P$ with coefficients in $\mathcal{U}(\mathbb{F})$ is completely even
if and only if
\begin{enumerate}
\item $P$ belongs to $\mathbb{F}[x]^{\mathrm{even}}$, and
\item $P$ factors only into odd factors which are units or even polynomials.
\end{enumerate}

\begin{theorem}
Suppose $\mathbb{F}$ is a prime field and that $P_0$ is any polynomial in $\mathbb{F}[x]^{\mathrm{even}}$.
Then $\mathbb{F}[x]\diagup\langle P_0\rangle$ is a unital 3-field if and only if $P_0$ is
completely even.
\end{theorem}
\begin{proof}
Clearly, whenever $P_0$ has a factorization $P_0=QP$, with $Q$ a non-invertible odd polynomial,
$\langle Q\rangle$ is an ideal larger than $\langle P_0\rangle$, strictly smaller than
$\mathbb{F}[x]^{\mathrm{even}}$, and intersecting $\mathbb{F}[x]^{\mathrm{odd}}$. So
$\mathbb{F}[x]\diagup\langle P_0\rangle$ cannot be a 3-field in light of Theorem \ref{th-maxi}.

For the converse, suppose $P_0$ is completely even. Additionally, we
assume that $P_0$ does not contain any invertible odd factor, not affecting the following argument.

We first look at the case in which $|\mathbb{F}|=\infty$.
For an ideal $\mathfrak{I}\in \mathcal{U}[x]$ we write $\mathfrak{I}_\mathbb{Q}$ for the ideal $\mathfrak{I}$ generates
in $\mathbb{Q}[x]$, i.e.\
\begin{equation}
    \mathfrak{I}_\mathbb{Q}=\set{2^{-n}P}{P\in \mathfrak{I},\ n\in \mathbb{N}}.
\end{equation}
Since $(\mathcal{U}[x]P_0)_\mathbb{Q}$ is the principal ideal generated by $P_0$ in $\mathbb{Q}[x]$,
we find for any ideal $\mathfrak{I}$ of (and different from) $\mathcal{U}[x]$, larger than $\mathcal{U}[x]P_0$,
a factorization $P_0=PQ$ so that $\mathfrak{I}_\mathbb{Q}=\mathbb{Q}[x]P$. Since for no $n\in \mathbb{N}$, $2^{-n}$ is
a factor of $P_0$, we actually may suppose that $P,Q\in \mathcal{U}[x]$ and still have
$\mathfrak{I}_\mathbb{Q}=\mathbb{Q}[x]P$. As $P_0$ was assumed to be completely even
(and not containing any factor that is a unit) $P$ and $Q$ have to be even and so
\begin{equation}
       \mathfrak{I}\cap \mathcal{U}[x]^{\mathrm{odd}}\subseteq
     \mathbb{Q}[x]P\cap \mathcal{U}[x]^{\mathrm{odd}}=\emptyset,
\end{equation}
proving that $\mathbb{F}[x]\diagup\langle P_0\rangle$ is a field, by Theorem \ref{th-maxi}.

Now suppose $|\mathbb{F}|=2^n$ and denote by
\begin{equation}
\pi_n: \mathbb{Q}^{\mathrm{odd}}[x]\to (\mathbb{Z}\diagup 2^n \mathbb{Z})^{\mathrm{odd}}[x]
\end{equation}
the canonical quotient map,
reducing the coefficients of elements in $\mathbb{Q}^{\mathrm{odd}}[x]$ to coefficients in
$(Z\diagup 2^n Z)^{\mathrm{odd}}$. We again will suppose that $P_0$ does not contain any invertible odd factor.
Fix a polynomial $P_1\in\mathbb{Q}^{\mathrm{even}}[x]$, of the same degree as $P_0$, such that $\pi_n(P_1)=P_0$.
The map $\pi_n$ decreases degrees of polynmomials, preserving the one of $P_1$, hence any factorization of the latter polynomial into non-even factors would result into one of the same kind for $P_0$. Therefore $P_1$ has to
be completely even as well, and so by the first part of the proof, $\mathbb{Q}^{\mathrm{odd}}[x]\diagup\langle P_1\rangle$ is a
unital 3-field.
The result then follows from the fact that for $|\mathbb{F}|=2^n$
\begin{equation}
    \mathbb{F}[x]\diagup\langle P_0\rangle=
    \pi_n\left(\mathbb{Q}^{\mathrm{odd}}[x]\diagup\langle P_1\rangle\right).
\end{equation}
\end{proof}

\begin{corollary}
Each finite, singly generated unital 3-field $\mathbb{F}$ is the quotient of a
singly generated unital 3-field $\widehat{\mathbb{F}}$ with prime field
$\mathbb{Q}^{\mathrm{odd}}$.
\end{corollary}

Factoring polynomials is no easy business under these circumstances. For the example below, we will use
\begin{lemma}
Let $R$ be a unital 3-ring and $p\in\mathcal{Q}(R)$.
If $\phi:R\to R'$ is a morphism respecting non-units and $\phi(p)$ is non-zero and completely even
  then $p$ is completely even.
\end{lemma}
\begin{proof}
The proof is easy: If there is a factorization of $p$ in which one of the factors is odd and not a unit,
it follows that the same type of factorization exists for $\phi(p)$, as $\phi$ respects
(non-units and) the distinction between even and odd.
\end{proof}
\begin{example}
Let $\mathbb{F}$ be any of the prime 3-fields and denote by $\phi: \mathbb{F}[x]\to \mathbb{F}_0[x]$ the morphism
reducing coefficients $\mod 2$. We claim that $\phi$ is a morphism as in the lemma:

Recall that in $R[x]$, $R$ a unital (binary) ring,
$P=a_0+\ldots+a_nx^n$ is a unit iff $a_0$ is a unit in $R$ and $a_\nu$ is nilpotent for all $\nu\geq 1$.
For a non-unit $P\in\mathbb{F}[x]$ then, either $a_0$ is even, in which case it is mapped to a non-unit
in $\mathbb{F}_0$, or, $a_0$ is a unit but one of the coefficients, $a_k$, $k>1$, is mapped to 1,
and so also in this case $\phi(P)$ is a non-unit.
\end{example}

\begin{example}
Let $\mathbb{F}$ be any of the prime 3-fields. Then $x^n-1$ is completely even iff $n$ is a power of two.
In fact, let $n=n_0 2^k$ with $n_0$ odd. Then
\begin{multline}
x^{n}-1=
\left(x^{n_0}-1\right)\left(x^{n_0}+1\right)\left(x^{2n_0}+1\right)\cdots
\left(x^{n_02^{k-1}}+1\right)=\\
=\left(x-1\right)\left(\sum_{m=0}^{n_{0}-1}x^m\right)
\prod_{j=0}^{k-1}\left(x^{2^jn_{0}}+1\right).
\end{multline}
If $n_0>1$, the polynomial $\sum_{m=0}^{n_{0}-1}x^m$ is odd (and not a unit),
so $x^n-1$ is not completely even. If, however, $n_0=1$ we invoke
the morphism $\phi$ from above that reduces coefficients to the (binary) field $\mathbb{Z}/2$.
Using the Frobenius morphism, for $n$ a power of 2, we have $x^n-1=(x-1)^n$,
and the claim follows.
\end{example}

We collect the above into
\begin{theorem}
Let $\mathbb{F}$ be any of the finite prime fields. Then, there is a splitting unital 3-field
for $\mathbb{F}$ and the polynomial $x^n-1$ if and only if $n$ is a power of 2.
\end{theorem}
Note that in classical (binary) field theory there is a splitting field for the finite
prime fields (essentially) for all the polynomials $x^n-1$.

Finally, we will consider the already quite difficult case of finite unital
3-fields for which the characteristic is one, $\chi\left(\mathbb{F}\right)=1$.

\begin{theorem}\label{Singly generated fields}
\label{theor-f}Suppose $\mathbb{F}$ is a finite unital $3$-field with $\chi(\mathbb{F})=1$,
and denote by $\mathbb{F}_0=\left\{ 1 \right\}$ its prime field. If $\mathbb{F}$ is generated by a single
element, then it is isomorphic to
$\mathbb{F}_0[x]\diagup\left\langle \left(x-1\right)^{n}\right\rangle$.
Furthermore,
\begin{equation}
\mathbb{F}_0[x]\diagup\left\langle \left(x-1\right)^{n}\right\rangle\cong\mathbb{F}_{0}(n),
\end{equation}
where $\mathbb{F}_{0}(n)$ was defined in Example~\ref{Definition F(n)}.
\end{theorem}

\begin{proof}
By the above result, and since $\mathcal{U}(\mathbb{F}_0)=\mathbb{Z}\diagup 2\mathbb{Z}$ is a (binary) field, there must be
a completely even polynomial $P\in(\mathbb{Z}\diagup 2\mathbb{Z})[x]$ so that
$\mathbb{F}=\mathbb{F}_0[x]\diagup\langle P\rangle$.
Writing a polynomial $Q\in(\mathbb{Z}\diagup 2\mathbb{Z})[x]$ in powers
of $(x-1)$ we have that $Q$ is even iff its constant term is even, and it is odd
iff its constant term is. Now, if the polynomial $P$ is not of the desired form,
there are $n_{0}<\ldots<n_{k}$ with
\begin{equation}
P=(x-1)^{n_{0}}+\ldots+(x-1)^{n_{k}}=(x-1)^{n_{0}}(1+\ldots+(x-1)^{n_{k}
-n_{0}}).
\end{equation}
Hence, such a polynomial is not completely even, and we are done.

Finally, $\mathbb{F}_0[x]\diagup\left\langle \left(x-1\right)^{n}\right\rangle\cong\mathbb{F}_{0}(n)$
is proven as follows. Consider the mapping $\Phi:\mathbb{F}_0[x]
\longrightarrow\mathbb{F}_0(n)$,
\begin{equation}
\Phi\left(Q\right)=
\left( Q\left( 1 \right),\ldots,\frac{Q^{\left( \nu \right)}\left( {1} \right)}{\nu!},
\ldots,\frac{Q^{\left( n-1 \right)}\left( {1} \right)}{(n-1)!} \right)\mod 2
\end{equation}
where only the coefficients of the respective polynomials in $\mathbb{F}_{0}(n)$ are written
down and (formal) derivatives are used. This map is a (surjective) field
morphism: it clearly is additive, while multiplicativity follows as the product
of Taylor polynomials of two given polynomials also in the present situation is the Taylor polynomial
of the their products.
Since the kernel of this map is equal to $\left\langle (x-1)^{n}\right\rangle$,
the second statement of the theorem follows.
\end{proof}

We look into some further examples and determine some automorphism groups.

\begin{lemma}
There is a 1-1 correspondence between the automorphisms of $\mathbb{F}_0(n)$ and
polynomials $P=\sum_{j} f_{j}x^{j}$, $f_{j}\in \mathcal{U}(\mathbb{F}_0(n))$,
$\sum f_{j}\in \mathbb{F}_0$,
 for which there exists another such polynomial $Q$
with $P\circ Q=Q\circ P=x $.
\end{lemma}
\begin{proof}
Suppose $\Psi$ is an automorphism of $\mathbb{F}_0(n)$. Let $P=\Psi(x)=\sum_{j} f_{j}x^{j}$.
Then $f_{j}\in \mathcal{U}(\mathbb{F}_0(n))$ and $\sum f_{j}\in \mathbb{F}_0$. Since the
polynomial $x$ is generating, $P$ uniquely determines $\Psi$. Furthermore,
$\Psi(\sum_{j} g_{j}x^{j})=\sum_{j} g_{j}P^{j}$, and so composition
of morphisms consists in substitution of polynomials. In particular, $\Psi^{-1}$
must correspond to a polynomial $Q$ with $P\circ Q=Q\circ P=x$.
\end{proof}
\begin{example}
Let us start with a classical field extension. For $\chi\left(  \mathbb{F}
\right)  =1$ we formally adjoin a square root of $3$ and obtain $\mathbb{F}
\left[  \sqrt{3}\right]  $. More explicitly,
\begin{equation}
\mathbb{F}\left[  \sqrt{3}\right]  =\left\{  a+b\sqrt{3}\mid a,b\in
\mathbb{Z}\diagup2\mathbb{Z},a+b=1\right\}  ,
\end{equation}
and it turns out that this $3$-field is isomorphic to $\mathbb{F}\left(
2\right)  $. Note that this field has a trivial automorphism group.
\end{example}

\begin{example}
Let us have a look at $\mathbb{F}\left(  3\right)  ,$ where $\mathbb{F}$ is as
before. We use polynomials in the generator $x$ and denote the elements $a=x$,
$b=x^{2}$, $c=x^{2}+x+1$, $a,b,c,1\in\mathbb{F}\left(  3\right)  $. The
multiplicative Cayley table is\textbf{ }

\begin{equation}
\begin{tabular}
[c]{||c||c|c|c|c||}\hline\hline
$\cdot$ & $1$ & $a$ & $b$ & $c$\\\hline\hline
$1$ &$ 1 $ & $ a $ & $ b $ & $ c$\\\hline
$a$ &$ a $ & $ b $ & $ c $ & $ 1$\\\hline
$b$ &$ b $ & $ c $ & $ 1 $ & $ a$\\\hline
$c$ &$ c $ & $ 1 $ & $ a $ & $ c$\\\hline\hline
\end{tabular}
\end{equation}
To find nontrivial automorphisms on $\mathbb{F}\left(  3\right)  $ we
construct the Cayley table under compositions as
\begin{equation}
\begin{tabular}
[c]{|c||c|c|c|}\hline
$\circ$ & $a$ & $b$ & $c$\\\hline\hline
$a$ & $a $ & $ b $ & $ c$\\\hline
$b$ &$ b $ & $ 1 $ & $ b$\\\hline
$c$ &$ c $ & $ b $ & $ a=x$\\\hline
\end{tabular}
\end{equation}
We observe that one nontrivial automorphism on $\mathbb{F}\left(  3\right)  $
is connected to $c$, because $c\circ c=a=x$. So the group of automorphisms is
$\mathbb{Z}\diagup 2\mathbb{Z}$, or, for later use, the dihedral group of order 2, respectively.
\end{example}

\begin{example}
Let us denote elements of $\mathbb{F}\left(  4\right)  $ as
\begin{equation}
\begin{tabular}
[c]{|c|c|c|c|c|c|c|}\hline
$a$ & $b$ & $c$ & $d$ & $e$ & $f$ & $g$\\\hline
$x$ & $x^{2}$ & $x^{3}$ & $x^{2}+x+1$ & $x^{3}+x+1$ & $x^{3}+x^{2}+1$ &
$x^{3}+x^{2}+x$\\\hline
\end{tabular}
\end{equation}
The multiplicative Cayley table for $\mathbb{F}\left(  4\right)  $ is
\begin{equation}
\begin{tabular}
[c]{||c||cccccccc||}\hline\hline
$\cdot $ & $1$ & $a$ & $b$ & $c$ & $d$ & $e$ & $f$ & $g$\\\hline\hline
$1$ & $ 1 $ & $ a $ & $ b $ & $ c $ & $ d $ & $ e $ & $ f $ & $ g$\\
$a$ & $a $ & $ b $ & $ c $ & $ 1 $ & $ g $ & $ d $ & $ e $ & $ f$\\
$b$ &$ b $ & $ c $ & $ 1 $ & $ a $ & $ f $ & $ g $ & $ d $ & $ e$\\
$c$ & $c $ & $ 1 $ & $ a $ & $ b $ & $ e $ & $ f $ & $ g $ & $ d$\\
$d$ & $d $ & $ g $ & $ f $ & $ e $ & $ b $ & $ a $ & $ 1 $ & $ c$\\
$e$ & $e $ & $ d $ & $ g $ & $ f $ & $ a $ & $ 1 $ & $ c $ & $ b$\\
$f$ & $f $ & $ e $ & $ d $ & $ g $ & $ 1 $ & $ c $ & $ b $ & $ a$\\
$g$ & $g $ & $ f $ & $ e $ & $ d $ & $ c $ & $ b $ & $ a $ & $ 1$\\\hline\hline
\end{tabular}
\end{equation}

To find nontrivial automorphisms we make use of the Cayley table w.r.t.
composition, and we find that the (nontrivial) automorphisms are exactly those
given by the polynomials $c,d,f$. Composition is then given by the following
table
\begin{equation}
\begin{tabular}
[c]{||c||cccc||}\hline\hline
$\circ$ & $a$ & $c$ & $d$ & $f$\\\hline\hline
$a$ & $a=x$ & $c $ & $ d $ & $ f$\\
$c$ &$ c $ & $ a=x$  & $ f $ & $ d$\\
$d$ &$ d $ & $ f $ & $ a=x$  & $ c$\\
$f$ & $f $ & $ d $ & $ c $ & $ a=x$\\\hline\hline
\end{tabular}
\end{equation}
which upon inspection yields that the automorphism group equals the dihedral
group of order 4.
\end{example}

\begin{example}
Let us denote elements of $\mathbb{F}\left(  5\right)  $ as
\begin{align}
&
\begin{tabular}
[c]{|c|c|c|c|c|c|c|}\hline
$a$ & $b$ & $c$ & $d$ & $e$ & $f$ & $g$\\\hline
$x$ & $x^{2}$ & $x^{3}$ & $x^{4}$ & $x^{2}+x+1$ & $x^{3}+x+1$ & $x^{4}
+x+1$\\\hline
\end{tabular}
\nonumber\\
&
\begin{tabular}
[c]{|c|c|c|c|}\hline
$p$ & $q$ & $r$ & $s$\\\hline
$x^{3}+x^{2}+1$ & $x^{4}+x^{2}+1$ & $x^{4}+x^{3}+1$ & $x^{3}+x^{2}+x$\\\hline
\end{tabular}
\nonumber\\
&
\begin{tabular}
[c]{|c|c|c|c|}\hline
$t$ & $u$ & $v$ & $w$\\\hline
$x^{4}+x^{2}+x$ & $x^{4}+x^{3}+x$ & $x^{4}+x^{3}+x^{2}$ & $x^{4}+x^{3}
+x^{2}+x+1$\\\hline
\end{tabular}
\end{align}
Then similar calculations give us the nontrivial automorphisms induced by the
polynomials $c,e,g,p,r,t,v$. Together with the identity morphism $a=x$ they
obey the following Cayley table under composition
\begin{equation}
\begin{tabular}
[c]{||c||cccccccc||}\hline\hline
$\circ$ & $a$ & $c$ & $e$ & $g$ & $p$ & $r$ & $t$ & $v$\\\hline\hline
$a$ & $a=x$ & $c $ & $ e $ & $ g $ & $ p $ & $ r $ & $ t $ & $ v$\\
$c$ & $c$ &$ a=x$ &$ p $ & $ r $ & $ e $ & $ g $ & $ v $ & $ t$\\
$e$  & $ e $ & $ v $ & $ g $ & $ t $ & $ c $ & $ p $ & $ a=x $ & $ r$\\
$g$  & $ g $ & $ r $ & $ t $ & $ a=x $ & $ v $ & $ c $ & $ e $ & $ p$\\
$p$  & $ p $ & $ t $ & $ r $&$ v$ & $a=x$ & e &$ c$ & $g$\\
$r$ &$ r$ & $g $ & $ v $ & $ c $ & $ t $ & $ a=x $ & $ p $ & $ e$\\
$t$ & $ t $ & $ p $ & $ a=x $ & $ e $ & $ r $ & $ v $ & $ g $ & $ c$\\
$v$  & $ v $ & $ e $ & $ c $ & $ p $ & $ g $ & $ t $ & $ r $ & $ a=x$\\\hline\hline
\end{tabular}
\ \ \label{autF}
\end{equation}

This is the dihedral group of order 8.
\end{example}

Note that all these automorphism groups would be the ``Galois groups''
for the respective extensions of \{1\}. Also,
$\mathbb{F}\left(  n\right)  $ is a quotient of $\mathbb{F}\left(  m\right)  $
whenever $n<m$.

The class of all unital 3-field of characteristic different from one might be inaccessible for
a (concrete) classification. Here are some examples, where the respective fields
require more than one generator in the present definition.

Fix a unital 3-field $\mathbb{F}$, and put
\begin{equation}
\mathbb{F}_0\left(  n_{1},\ldots,n_{k}\right)=\mathbb{F}_0\left[  x_{1}
,\ldots,x_{k}\right] \diagup\left\langle \left(
x_{1}-1\right)  ^{n_{1}},\ldots,\left(  x_{k}-1\right)  ^{n_{k}}\right\rangle
\end{equation}
This extension of $\mathbb{F}$ is characterized by the fact that it displays
the fewest possible relations a field of $k$ generators possibly can have
(this will be made precise below). It can be shown that
\begin{multline}
\mathbb{F}_0\left(  n_{1},\ldots,n_{k}\right)  =\\
\left\{  1+\sum_{(1,\ldots,1)\leq\alpha\leq(n_{1}-1,\ldots,n_{k}-1)}
\varepsilon_{\alpha}(1-x)^{\alpha}\mid\varepsilon_{\alpha}=0,1,\ (x_{k}
-1)^{n_{k}}=0,\ k=1,\ldots,n\right\}  ,
\end{multline}
where $(1-x)^{\alpha}=(1-x_{1})^{\alpha_{1}}\cdots(1-x_{n})^{\alpha_{n}}$.
Much more relations are necessary in order to present the Cartesian product
$\mathbb{F}_0\left(  n_{1}\right)  \times\ldots\times\mathbb{F}_0\left(
n_{k}\right)  $. Denote by $\xi_{i}$ the generator of the 3-field
$\mathbb{F}_0(n_{i})$ and by $x_{i}$ the element of $\mathbb{F}_0\left(
n_{1}\right)  \times\ldots\times\mathbb{F}_0\left(  n_{k}\right)  $ which has
the unit element in each entry except at the place $i$ where it is $\xi_{i}$.
Then for each element $(f_{1},\ldots,f_{n})\in\mathbb{F}_0\left(  n_{1}\right)
\times\ldots\times\mathbb{F}_0\left(  n_{k}\right)  $ there are $\varepsilon
_{ij}=0,1$ so that
\begin{equation}
(f_{1},\ldots,f_{n})=(1,\dots,1)+\sum_{i=1}^{n}\sum_{j=1}^{k_{n}}
\varepsilon_{ij}(1-x_{i})^{j}.
\end{equation}
Consequently,
\begin{multline}
\mathbb{F}_0\left(  n_{1}\right)\times\ldots\times\mathbb{F}_0\left(
n_{k}\right)  =\\
\mathbb{F}_0\left[  x_{1},\ldots,x_{k}\right] \diagup
\left\langle \left(  x_{1}-1\right)  ^{n_{1}},\ldots,\left(  x_{k}-1\right)
^{n_{k}}, (x_{i}-1)(x_{j}-1), i,j=1,\ldots,n, i\neq j\right\rangle
\end{multline}

\begin{theorem}
Let $\mathbb{F}$ be a finite field with $\chi(\mathbb{F})=1$, generated as a
unital 3-field by $n$ elements. Let, as before, $\mathbb{F}_{0}=(\mathbb{Z}
/2\mathbb{Z})^{\mathrm{odd}}=\{1\}$.

\begin{enumerate}
\item There exist natural numbers $k_{1},\ldots,k_{n}$ such that $\mathbb{F}$
is a quotient of $\mathbb{F}_{0}(k_{1},\ldots,k_{n})$.

\item The ideal $\mathcal{J}$ such that $\mathbb{F}\cong\mathbb{F}_0\left[
x_{1},\ldots,x_{n}\right]  \diagup\mathcal{J}$ is of the form
\begin{equation}
\mathcal{J}=\left\langle \left(  x_{1}-1\right)  ^{k_{1}},\ldots,\left(
x_{n}-1\right)  ^{k_{n}},P_{1},\ldots,P_{N}\right\rangle , \label{jp}
\end{equation}
where the polynomials $P_{1},\ldots, P_{N}$ are neither divisible by an odd polynomial
nor by any of the $(x_{k}-1)^{n_{k}}$.
\end{enumerate}
\end{theorem}

\begin{proof}
Denote by $x_{1},\ldots,x_{n}$ the 3-field generators of $\mathbb{F}$. As each
of them generates a unital 3-field, $(x_{i}-1)^{k_{i}}=0$ for some $k_{i}$,
$i=1,\ldots,n$, and it follows that there is a quotient map of $\mathbb{F}
_{0}(k_{1},\ldots,k_{n})$ onto $\mathbb{F}$.

In order to prove the second part of the theorem, we select even polynomials
$P_{1},\ldots,P_{N}$, not divisible by any of the $(x_{k}-1)^{n_{k}}$, so that
$\mathbb{F}\cong\mathbb{F}_{0}(k_{1},\ldots,k_{n})
\diagup\langle P_{1},\ldots,P_{N}\rangle$. Similar to the proof of
{Theorem \ref{theor-f}}, one can show that it is not possible that any of
these polynomials contains an odd factor (Alternatively, one can use the fact
that all odd polynomials which can arise as factors here are invertible.)
\end{proof}

\begin{example}
Let us consider the \textquotedblleft unfree\textquotedblright\ unital
$3$-field $\mathbb{F}^{2}$ and the \textquotedblleft free\textquotedblright
\ unital $3$-field $\mathbb{F}\left(  n_{1},n_{2}\right)  $. We show that
$\mathbb{F}\left(  2\right)  \times\mathbb{F}\left(  2\right)  $ and
$\mathbb{F}\left(  2,2\right)  $ are not isomorphic. By definition, we have
$\mathbb{F}\left(  2\right)  =\left\{  1+\varepsilon\left(  y-1\right)
\mid\varepsilon\in\mathbb{Z}\diagup 2\mathbb{Z}\right\}  $, which contains $2$ elements
$\left\{  1,y\right\}  $ with the relations (as pairs) $1+1=y+y=0$ and
$y^{2}=1$. The most unfree $3$-field $\mathbb{F}\left(  2\right)
\times\mathbb{F}\left(  2\right)  $ has $4$ elements and generated by
$x_{1}=\left(
\begin{array}
[c]{c}
y\\
1
\end{array}
\right)  $, $x_{2}=\left(
\begin{array}
[c]{c}
1\\
y
\end{array}
\right)  $. It is easily seen that $x_{1}^{2}=x_{2}^{2}=1$, $x_{1}x_{2}
=x_{1}+x_{2}-1$, and therefore
\begin{equation}
\mathbb{F}\left(  2\right)  \times\mathbb{F}\left(  2\right)  =\mathbb{Z}
_{2}\left[  x_{1},x_{2}\right]  ^{\mathrm{odd}}\diagup\left\langle \left(
x_{1}-1\right)  ^{2},\left(  x_{2}-1\right)  ^{2},P\right\rangle , \label{fp}
\end{equation}
where the additional polynomial is $P=\left(  x_{1}-1\right)  \left(
x_{2}-1\right)  $ (see (\ref{jp})). On the other hand,
\begin{equation}
\mathbb{F}\left(  2,2\right)  =\left\{  1+\varepsilon_{1}\left(
x_{1}-1\right)  +\varepsilon_{2}\left(  x_{2}-1\right)  +\varepsilon\left(
x_{1}-1\right)  \left(  x_{2}-1\right)  \mid\varepsilon_{i}\in\mathbb{Z}
_{2}\right\}
\end{equation}
contains $8$ elements and
\begin{equation}
\mathbb{F}\left(  2,2\right)  =\mathbb{Z}\diagup 2\mathbb{Z}\left[  x_{1},x_{2}\right]
^{\mathrm{odd}}\diagup\left\langle \left(  x_{1}-1\right)  ^{2},\left(
x_{2}-1\right)  ^{2}\right\rangle ,
\end{equation}
which is not isomorphic to the field in (\ref{fp}).
\end{example}

Our final examples show that there exist noncommutative finite unital 3-fields.

\begin{example}
[Triangular $3$-fields]\label{ex-tri}
Let $\mathbb{F}$ be a unital 3-field and put
\begin{equation}
\mathbb{D}\left(  n,\mathbb{F}\right)  =\left\{  A_{n}\left(  f_{i},b_{ij}\right)  \mid
f_{i}\in\mathbb{F},b_{ij}\in\mathcal{U}\left(  \mathbb{F}\right)
,i.j=1,\ldots,n\right\}  ,
\end{equation}
where
\begin{equation}\label{5.26}
A_{n}\left(  f_{i},b_{ij}\right)  =\left(
\begin{array}
[c]{ccc}
f_{1} &  & 0\\
& \ddots & \\
b_{ij} &  & f_{n}
\end{array}
\right)
\end{equation}
In order to see that $\mathbb{D}\left(n,\mathbb{F}\right)$ is a unital 3-field, observe that
sums and products of two such matrices lie again in $\mathbb{D}\left(  n,\mathbb{F}\right)$ and $\mathbb{D}\left(  n,\mathbb{F}\right)  $ is a noncommutative
unital $3$-algebra (if $n>1$). Furthermore, each matrix (\ref{5.26}) is invertible, and its inverse has the form
\begin{equation}
A_{n}\left(f_{i},b_{ij}\right)^{-1}=A_{n}\left(f_{i}^{-1},\hat{b}_{ij}\right)=
\left(
\begin{array}
[c]{ccc}
f_{1}^{-1} &  & 0\\
& \ddots & \\
\hat{b}_{ij} &  & f_{n}^{-1}
\end{array}
\right)  ,
\end{equation}
where $\hat{b}_{ij}\in\mathcal{U}\left(  \mathbb{F}\right)  $. This follows from the fact that
the standard procedure for the inversion of a matrix yields expressions
 which are well-defined within $\mathbb{F}$. In total,
$\mathbb{D}\left(  n,\mathbb{F}\right)  $ is a noncommutative unital
$3$-field, which is finite in case $\mathbb{F}$ is.
\end{example}

\begin{example}
[Quaternion $3$-fields]We start by selecting a unital 3-field
$\mathbb{F}$, and will equip the free 3-vector space (with basis elements $i_{\mu}$, $\mu=0,1,2,3$)
\begin{equation}
\mathbb{HF}=\left(\mathbb{F}^{4}\right)^{\mathrm{free}}=\left\{\sum_{\mu=0}^{3}
a_{\mu}i_{\mu}\mid a_{\mu}\in\mathcal{U}\left(  \mathbb{F}\right)
,\sum_{\mu=0}^{3}a_{\mu}\in\mathbb{F}\right\}  ,
\end{equation}
with a multiplication (so that it will become a unital 3-field) in the following way:
We assume the quaternion relations
\begin{equation}
i_{1}^{2}=i_{2}^{2}=i_{3}^{2}=i_{1}i_{2}i_{3}=-1,i_{0}=1
\end{equation}
and linearly extend them to a multiplication on the whole of
$\mathbb{HF}$. This product is well-defined, since the sum of coefficients
$c_k\in\mathcal{U}(\mathbb{F})$
in
\begin{equation}
\left(\sum_{k=0}^{3}a_ki_k\right)\left(\sum_{k=0}^{3}b_ki_k\right)=\sum_{k=0}^{3}c_ki_k,
\end{equation}
where $a_k,b_k\in\mathcal{U}(\mathbb{F})$,
$\sum_{k=0}^{3}a_k\in \mathbb{F},\sum_{k=0}^{3}b_k\in \mathbb{F}$ can be written in the form
\begin{equation}
\sum_{k=0}^{3}c_k=
\left(\sum_{k=0}^{3}a_k\right)\left(\sum_{k=0}^{3}b_k\right)-2\gamma,\quad
\gamma\in\mathcal{U}(\mathbb{F}).
\end{equation}
Since, in general, for an element $u\in\mathcal{U}(\mathbb{F})$, we have that
$2u\in\mathcal{Q}(\mathbb{F})$, it follows that the sum over the $c_k$'s is in $\mathbb{F}$,
and $\mathbb{HF}$ is a unital 3-algebra. In
addition, if $q=a_{0}+a_{1}i_{1}+a_{2}i_{2}+a_{3}i_{3}$, we let $\bar{q}
=a_{0}-a_{1}i_{1}-a_{2}i_{2}-a_{3}i_{3}$ and observe that $q\bar{q}=(a_{0}
^{2}+a_{1}^{2}+a_{2}^{2}+a_{3}^{2})i_{0}\in\mathbb{HF}$. Then,
for each element q of $\mathbb{HF}$, an inverse is given by
\begin{equation}
q^{-1}=\dfrac{\bar{q}}{q\bar{q}}=\dfrac{a_{0}-a_{1}i_{1}-a_{2}i_{2}-a_{3}
i_{3}}{a_{0}^{2}+a_{1}^{2}+a_{2}^{2}+a_{3}^{2}}\in\mathbb{HF}
\end{equation}
and thus, $\mathbb{HF}$ is a non-commutative 3-field. (In this construction, the elements
$i_k$ can be replaced by the canonical generators/relations of Clifford algebras in order to obtain
a wider class of examples.)
\end{example}
Why is the above construction not possible if $\mathbb{F}$ is replaced by a finite binary field?
The main reason is that in ternary fields there are no zero elements. In fact, each finite binary field, for
some prime $p$, $\mathbb{Z}/p$ is a subfield, and now the probably fastest way to find an obstruction is through
Lagrange's four-square theorem, by which each integer, hence also the prime $p$, is the sum of four squares.

\vskip 1cm

\noindent
\textsc{Acknowledgements.} Research supported by the ERC through  AdG 267079,
as well as by Deutsche Forschungsgemeinschaft (DFG) within the framework of
\emph{Exzellenzstrategie des Bundes und der L\"{a}nder EXC 2044–390685587, Mathematik M\"{u}nster:
Dynamik –- Geometrie –- Struktur}.

The first author (S.D.) would like to express his gratitude
to Joachim Cuntz and Raimar Wulkenhaar for financial support,
at a very early stage of this work.

\end{document}